\documentclass[12pt,a4paper]{amsart}

\usepackage[all]{xy}
\usepackage{amsmath,amssymb,amsthm,amsfonts}
\usepackage{amscd}
\allowdisplaybreaks
\usepackage[top=35mm, bottom=25mm, left=25mm, right=25mm]{geometry}
\usepackage[pagebackref=true,colorlinks,linkcolor=magenta,citecolor=blue]{hyperref}
\usepackage{tikz}
\usepackage{enumitem}

\begin{document}

\title{Rigid ideals by deforming quadratic letterplace ideals}
\author{Gunnar Fl{\o}ystad}
\address{Universitet i Bergen, Matematisk Institutt, Postboks 7803, 5020, Bergen, Norway}
\email{gunnar@mi.uib.no}
\author{Amin Nematbakhsh}
\address{Institute for Research in Fundamental Sciences (IPM), Tehran, Iran}
\email{nematbakhsh@ipm.ir}
\date{\today}

\begin{abstract} 
We compute the deformation space of quadratic letterplace ideals $L(2,P)$ 
%(the same as edge ideals of Cohen-Macaulay bipartite graphs)
of finite 
posets $P$ when its Hasse diagram is a rooted tree. These deformations 
are unobstructed. 
The deformed family has a polynomial ring as the base ring.
The ideal $J(2,P)$ defining the full family of deformations is a rigid ideal
and we compute it explicitly.
In simple example cases $J(2,P)$ is the ideal of maximal minors of a
generic matrix, the Pfaffians of a skew-symmetric matrix, and a ladder
determinantal ideal.
\end{abstract}

%\tableofcontents
\maketitle

\newtheorem{theorem}{Theorem}[section]
\newtheorem{proposition}[theorem]{Proposition}
\newtheorem*{proposition2}{Proposition}
\newtheorem{lemma}[theorem]{Lemma}
\newtheorem{corollary}[theorem]{Corollary}
\newtheorem{claim}[theorem]{Claim}
\newtheorem{conjecture}[theorem]{Conjecture}
\theoremstyle{definition}
\newtheorem{definition}[theorem]{Definition}
\newtheorem{remark}[theorem]{Remark}
\newtheorem{remarks}[theorem]{Remarks}
\newtheorem{terminology}[theorem]{Terminology}
\newtheorem{notation}[theorem]{Notation}
\newtheorem{notations}[theorem]{Notations}
\newtheorem{question}[theorem]{Question}
\newtheorem{questions}[theorem]{Questions}
\newtheorem{problem}[theorem]{Problem}
\newtheorem{problems}[theorem]{Problems}
\newtheorem{example}[theorem]{Example}
\newtheorem{examples}[theorem]{Examples}
\newtheorem{cexample}[theorem]{Counter example}
\newtheorem{cexamples}[theorem]{Counter examples}
\newtheorem{summary}[theorem]{Summary}
\newtheorem{construction}[theorem]{Construction}
\newtheorem{situation}[theorem]{Situation}
\newtheorem{fact}[theorem]{Fact}
\newtheorem{facts}[theorem]{Facts}
\newtheorem{exercise}[theorem]{Exercise}
\newtheorem{convention}[theorem]{Convention}
\newtheorem{algorithm}[theorem]{Algorithm}

%commands

%Fonts
\newcommand{\mfa}{\mathfrak{a}}
\newcommand{\mfm}{\mathfrak{m}}
\newcommand{\bfa}{\mathbf{a}}
\newcommand{\bfb}{\mathbf b}
\newcommand{\bfd}{\mathbf d}
\newcommand{\bfk}{\mathbf{k}}
\newcommand{\bft}{\mathbf{t}}
\newcommand{\bfx}{\mathbf x}
\newcommand{\bfu}{\mathbf u}
\newcommand{\ovbft}{\overline{\mathbf{t}}}
\newcommand{\ovbfx}{\overline{\mathbf x}}
\newcommand{\ovbfu}{\overline{\mathbf u}}
\newcommand{\mcF}{\mathcal{F}}
\newcommand{\mcI}{\mathcal{I}}
\newcommand{\mcN}{\mathcal{N}}
\newcommand{\mcO}{\mathcal{O}}
\newcommand{\al}{\alpha}
\newcommand{\be}{\beta}
\newcommand{\ga}{\gamma}

%Fields
\newcommand{\kk}{\kr}
\newcommand{\ZZ}{\mathbb{Z}}
\newcommand{\NN}{\mathbb{N}}
\newcommand{\RR}{\mathbb{R}}
\newcommand{\CC}{\mathbb{C}}

%My commands
\newcommand{\polynomialRing}{\kk[x_1,\ldots,x_n]}
\newcommand{\projectivespace}{\mathbf{P}^n_\kk}
\newcommand{\hilbertscheme}{\mathbf{H}}
\newcommand{\exactsequence}[3]{0\to #1 \to #2 \to #3 \to 0}
\newcommand{\colRed}[1]{\textcolor{red}{#1}}

%operators

%general
\newcommand{\id}{\operatorname{id}}
\newcommand{\im}{\operatorname{Im}}
\newcommand{\coker}{\operatorname{Coker}}
\renewcommand{\ker}{\operatorname{Ker}}

%category theory
\newcommand{\ob}{\operatorname{Ob}}
\newcommand{\mor}{\operatorname{Mor}}

%Commutative algebra
\newcommand{\height}{\operatorname{height}}
\newcommand{\length}{\operatorname{length}}
\newcommand{\spec}{\operatorname{Spec}}
\newcommand{\proj}{\operatorname{Proj}}
\newcommand{\Proj}{\mathbf{Proj}}
\newcommand{\V}{\operatorname{V}}
\newcommand{\supp}{\operatorname{Supp}}
\newcommand{\ann}{\operatorname{ann}}
\newcommand{\ass}{\operatorname{Ass}}
\newcommand{\codim}{\operatorname{codim}}
\newcommand{\transcendentaldegree}{\operatorname{tr.deg}}
\newcommand{\End}{\operatorname{End}}
\newcommand{\rank}{\operatorname{rank}}
\newcommand{\depth}{\operatorname{depth}}
\newcommand{\syz}{\operatorname{syz}}
\newcommand{\Mon}{\operatorname{Mon}}
\newcommand{\lcm}{\operatorname{lcm}}
\newcommand{\Min}{\operatorname{Min}}

%homological algebra
\newcommand{\Hom}{\operatorname{Hom}}
\newcommand{\ext}{\operatorname{Ext}}
\newcommand{\Tor}{\operatorname{Tor}}
\newcommand{\aut}{\operatorname{Aut}}
\renewcommand{\H}{\operatorname{H}}
\newcommand{\R}{\operatorname{R}}
\newcommand{\colim}{\operatorname{colim}}
\newcommand{\pd}{\operatorname{proj.dim}}

%sheaf theory
\renewcommand{\L}{\operatorname{L}}
\newcommand{\sheafkernel}{\mathfrak{Ker}}
\newcommand{\presheafcokernel}{\mathfrak{PCoker}}
\newcommand{\sheafcokernel}{\mathfrak{SCoker}}
\newcommand{\presheafimage}{\mathfrak{PIm}}
\newcommand{\sheafimage}{\mathfrak{SIm}}
\newcommand{\sheafhom}{\mathcal{H}\mathfrak{om}}
\newcommand{\sheafext}{\mathcal{E}\mathfrak{xt}}

%algebraic geometry
\newcommand{\D}{\operatorname{D}}
\newcommand{\picard}{\operatorname{Pic}}
\newcommand{\divisor}{\operatorname{Div}}
\newcommand{\cl}{\operatorname{Cl}}
\newcommand{\germ}{\operatorname{germ}}
\newcommand{\K}{\operatorname{K}}
\newcommand{\Sec}{\operatorname{Sec}}
\newcommand{\Tan}{\operatorname{Tan}}
\newcommand{\Z}{\operatorname{Z}}

%algebraic topology
\newcommand{\sing}{\operatorname{sing}}

%PhdThesis
\newcommand{\pw}{\mathbb P(W)}
\newcommand{\tatef}{\mathbf T^\bullet(\sheaf F)}
\newcommand{\tatepf}{\mathbf T^p(\sheaf F)}
\newcommand{\Fc}{\sheaf F}
\newcommand{\Sum}{\operatorname{Sum}}
\newcommand{\shifts}{\operatorname{Shifts}}
\newcommand{\mdeg}{\operatorname{mdeg}}
\newcommand{\Ex}{\operatorname{Ex}}
\newcommand{\Der}{\operatorname{Der}}
\newcommand{\In}{\operatorname{in}}

\newcommand{\kr}{\Bbbk}
\newcommand{\rarr}{\rightarrow}
\newcommand{\larr}{\leftarrow}
\newcommand{\ini}{\text{in}}
\newcommand{\te}{\otimes}
\newcommand{\iso}{\cong}
\newcommand{\Spec}{\text{Spec}\,}
\newcommand{\mto}[1]{\stackrel{#1}\longrightarrow}
\newcommand{\Yh}{\hat{Y}}
\newcommand{\Uh}{\hat{U}}
\newcommand{\Ih}{\hat{I}}

\newcommand{\gO}{{\mathcal O}}
\newcommand{\gA}{{\mathcal A}}

\newcommand{\Hilb}{\text{Hilb}}
\newcommand{\sus}{\subseteq}
\newcommand{\larrow}{\longrightarrow}
\newcommand{\comment}[1]{{\color{blue} \sf ($\clubsuit$ #1 $\clubsuit$)}}
\definecolor{awesome}{rgb}{1.0, 0.13, 0.32}
	\definecolor{darkspringgreen}{rgb}{0.09, 0.45, 0.27}
	\definecolor{darkpastelgreen}{rgb}{0.01, 0.75, 0.24}
\newcommand{\amincomment}[1]{{\color{darkspringgreen} #1}}

% Added by Gunnar
\newcommand{\ap}{{a^\prime}}
\newcommand{\sib}{\text{sib}}
\newcommand{\rot}{\rho}
\newcommand{\ztop}{{\mathbb Z} ([2] \times P)}
\newcommand{\bi}{{\mathbf i}}
\newcommand{\bk}{{\mathbf k}}
% Notation in Deformation section

\newcommand{\kalg}{\underline{\kr-\text{Alg}}}
\newcommand{\kvect}{\underline{\kr-\text{vect}}}
\newcommand{\kartalg}{\underline{\kr-\text{Art}}}
\newcommand{\sets}{\underline{\text{Set}}}
\newcommand{\kV}{\kr[V^*]/(V^*)^2}
\newcommand{\Def}[1]{\text{Def}_{#1}}
\newcommand{\ov}[1]{\overline{#1}}
\newcommand{\empt}{\emptyset}
\newcommand{\cA}{{\mathcal A}}
\newcommand{\tJ}{\tilde{J}}
\newcommand{\mm}{{\mathfrak m}}
\newcommand{\hY}{{\hat Y}}
\newcommand{\hbeta}{\hat{\beta}}
\newcommand{\oY}{\overline{Y}}
\newcommand{\oJ}{\overline{J}}

\section{Introduction}
Monomial ideal theory has much developed into a branch of its own.
But before that one studied polynomial ideals in general. Monomial  ideals came
about since they are specializations, typically initial ideals,
of such ideals. One should then 
ask: Can monomial ideal theory give something back? 
Can one {\it start} with monomial ideals 
and derive interesting classes of polynomial ideals in general? 
Yes one can, and here we do this for a
reasonably large class of monomial ideals.
We get a full understanding of the polynomial ideals
which specialize to the monomial ideals we start out from. 

\medskip
\noindent{\it The ideals we work with.} More precisely we
consider quadratically generated letterplace ideals $L(2,P)$
associated to a finite poset $P$.
These are precisely the edge ideals of Cohen-Macaulay bipartite graphs. 
Its generators are the monomials $x_{1,p}x_{2,q}$ where $p \leq q$ 
in the poset $P$. That edge ideals of Cohen-Macaulay bipartite graphs
have this form, is an astonishing discovery of J.Herzog and T.Hibi
\cite{HeHiCMbi}. This class of ideals were generalized in
\cite{EHM} and further studied and generalized in \cite{FGH}
were they were called letterplace ideals, see Section \ref{sec:LP}.

\medskip
\noindent{\it Results.} 
When the Hasse diagram of $P$ has the form of a rooted tree we
get a complete algebraic understanding of all ideals which are
deformations of the quadratic letterplace ideals $L(2,P)$. This is all the more
unusual and suprising for the following reason: Monomial ideals
are degenerations of polynomial ideals. Thus whenever a monomial ideal is on
a Hilbert scheme, it tends to be a singular point on the Hilbert
scheme. Its infinitesimal deformations are then obstructed and
it is a very hard and messy task to compute the space of all deformations.

However for the letterplace ideals $L(2,P)$ we consider, 
it turns out that every 
nice thing one could wish for, actually happens:
\begin{itemize}
\item The ideals $L(2,P)$ are
{\it unobstructed}, i.e. every infinitesimal deformations lifts. 
In  particular whenever this ideal is on a Hilbert scheme,
it is a smooth point.
\item The full deformation space which a priori is defined only
over a complete local ring, acutally lifts to a deformation over 
a {\it polynomial ring}. This follows from our computation of the
first cotangent cohomology of $L(2,P)$, Corollary \ref{cor:CotangRP},
and the explicit family we give in Section \ref{sec:Family}.
\item The full family of deformations over this polynomial ring 
is defined by a {\it rigid} ideal $J(2,P)$, Corollary \ref{cor:Rigid}. 
So deformations of $L(2,P)$ 
come from a coordinate change in $J(2,P)$.
\item We {\it explicitly} compute the ideal $J(2,P)$ by a simple recursive
procedure, see Section \ref{sec:Family} and Equation \ref{def:FamilyST}.
\end{itemize}

\noindent{\it Deforming Borel-fixed ideals.} 
As said monomial ideals are usually obstructed. 
Any ideal can be specialized to a Borel-fixed monomial ideal, which in
characteristic zero is the same as a strongly stable ideal.
One could then envision a path to classify ideals by deforming strongly
stable ideals. Unfortunately this is rather hopeless since strongly
stable ideals typically are very obstructed. (A notable exception
to this is the lexsegment ideal. When the polynomial
ring is given the standard grading the lexsegment ideal is a smooth point
on the Hilbert scheme \cite{Reev}, thus giving a distinguished
component of the Hilbert scheme for every Hilbert function $h : \ZZ \rarr \NN$.)

\medskip
\noindent{\it Deforming Stanley-Reisner ideals.}
Deformation theory applied to Stanley-Reisner
ideals has been developed by K.Altman and J.Christophersen. 
In \cite{AltCh-Cotang} they give the basic deformation theory 
for Stanley-Reisner rings. In \cite{AltCh-DefoSR} and \cite{Ch}
they consider triangulations of spheres, which deform to Calabi-Yau manifolds, 
and triangulations of tori, which deform to Abelian varieties. For classes of
triangulations they compute the versal deformation space (base space) of the 
(infinitesimal) deformation functor. This space is typically not
smooth, i.e. typically not a power series ring, 
but they give equations for the relations of this space, and give
a detailed description of it.
Here, for quadratic letterplace ideals,  we find that the base space both is
smooth and global and that we can give explicit equations for the
whole family of deformations. 
Recently \cite{ABHL} applied the theory developed by Christophersen and
Altman to investigate when monomial ideals are rigid. For edge ideals
they develop a number of results for when this holds. They also classify
the (few) letterplace ideals which are rigid.

\medskip
\noindent{\it Rigid ideals.} 
The notion of a rigid ideal occurs in (infinitesimal) deformation
theory. Although it is a well-known notion, we have not been able
to find many examples of rigid ideals in the literature. 
Classically determinantal ideals
of generic matrices are known to be rigid, \cite{BrVe}. Recently 
\cite{ChIl} shows that the coordinate rings of Grassmannians for the
Pl\"ucker embedding are rigid ideals. As mentioned above \cite{ABHL}
also gives classes of rigid monomial ideals. With the present article
we therefore believe we make a substantial contribution to the known
classes of rigid ideals.

\medskip
\noindent{\it Multigraded Hilbert schemes.}
While rigidity is an infinitesimal notion, one obtains global families
of deformations, the (multigraded) Hilbert schemes, when one endows the ambient
polynomial ring with a grading by an abelian group $A$ \cite{HaSt}. 
The $A$-graded
infinitesimal deformations are a subset of all infinitesimal deformations.
For global families there is then a situation close to rigidity. An ideal $I \sus \kr[X]$ 
may not be rigid, but there nevertheless is a rigid ideal $J$ in a larger
polynomial ring $\kr[Y]$ such that any deformation of $I$ comes from
a coordinate change in $J$ and then restricting to $\kr[X]$. 
J.Kleppe \cite{Kl} shows that this is the case when $I$ is a determinantal ideal
associated to a matrix of linear forms, but where there may be
dependencies between the linear entries. We show that the same
phenomenon happens here when we consider the letterplace ideals $I = L(2,P)$
and their $A$-graded deformations, Theorem \ref{thm:HilbSmooth} and the
applications after it.

\medskip
\noindent{\it Organization of the paper.} In Section 2 we recall 
letterplace ideals as defined in \cite{FGH}. In this article we are concerned
with the quadratic letterplace ideals $L(2,P)$ where the Hasse diagram of $P$ is
a rooted tree. In Section 3 we give
an explicit recursive procedure for computing the family $J(2,P)$ 
of deformations
of $L(2,P)$. Section 4 contains examples of these deformed families
for posets $P$ of cardinality $3$ and $4$, and also for two other 
simple classes of posets. In the first cases we get the ideal
of two-minors of a $2 \times 5$-matrix and the ideal of Pfaffians of
a skew-symmetric $5 \times 5$-matrix. 
Section 5 shows that the family of ideals $J(2,P)$ is very finely graded,
by a free abelian group of cardinality $2|P|$.
Section 6 investigates the deformation
theory of the ideals $L(2,P)$ and we compute the non-trivial first
order deformations of $L(2,P)$ for any finite poset $P$.
These are given by the first cotangent cohomology group.
This module turns out to have an extremely nice set
of generators. For each generator there is a single monomial $x_{1,p}x_{2,p}$
mapping to another monomial while all other monomials map to zero.
Section 7 shows flatness of $J(2,P)$ over the base polynomial ring.
In Section 8 we show rigidity of $J(2,P)$ (an infinitesimal notion).
Section 9 considers global families of deformations.
We show that the letterplace ideals $L(2,P)$ are smooth points on the
Hilbert schemes and that the general point on the Hilbert scheme comes
from the ideal $J(2,P)$ by a coordinate change and then restricting.
Sections 8 and 9 are developed in a general setting, and the results
concerning $L(2,P)$ are just particular instances of general results.
In the end we give Conjecture \ref{con:HilbCon}, that the results
of this article holds for any finite poset $P$ and not just
for posets whose Hasse diagram is a rooted tree.

\medskip
\noindent{\it Acknowledgement.} We thank Jan Christophersen for
useful discussions which significantly influenced the form of this paper.

\section{Letterplace ideals of posets}
\label{sec:LP}

Let $\kr$ be a field.
If $R$ is a set, denote by $\kr[x_R]$ the polynomial ring 
$\kr[x_i]_{i \in R}$.
% and if $S \sus R$ denote by $m_S$ the squarefree
%monomial $\Pi_{i \in S} x_i$. 
For a natural number $n$ let the chain poset be
$[n] = \{1 < 2 < \cdots < n \}$, so $[2] = \{1,2\}$. 

Given a finite poset $P$. We shall in this paper be concerned with 
the monomial ideal $L(2,P)$ in $\kr[x_{[2]\times P]}]$ 
generated by quadratic monomials $x_{1,p}x_{2,q}$ where $p \leq q$. 
These ideals are by \cite{HeHiCMbi} precisely the edge ideals of
Cohen-Macaulay bipartite graphs, see Section 9 of \cite{HeHiMon} for more
on this.
The ideals $L(2,P)$ are a special case of 
{\it letterplace ideals} $L(n,P)$ introduced in \cite{EHM} and 
\cite{FGH}, generated by monomials
\[ x_{1,p_1} x_{2,p_2}, \cdots, x_{n,p_n} \]
where $p_1 \leq \cdots \leq p_n$ are weakly increasing chains in $P$.
By  \cite[Corollary 2.5]{EHM}, see also 
\cite[Corollary 2.4]{FGH}, $L(n,P)$ is a Cohen-Macaulay ideal
of codimension equal to the cardinality $|P|$.
The multiplicity of $L(2,P)$ is the cardinality of the distributive
lattice of poset ideals of $P$, see Section 2 of \cite{DFN-LP}.  
This is the same as the degree of the algebraic subscheme of 
the affine space ${\mathbb A}^{2|P|}$ defined by $L(2,P)$.

For ease of notation
we shall rather write a variable $x_{i,p}$ as $p_i$. Thus
$L(2,P)$ is generated by quadratic monomials $p_1q_2$ where $p \leq q$.

\medskip
For more on letterplace ideals $L(n,P)$ and their Alexander duals 
$L(P,n)$ and the
omnipresence of these monomial ideals, see \cite{FGH}.
In \cite{DFN-LP} we compute the Betti tables of the letterplace ideals
$L(n,P)$, in particular of $L(2,P)$.

\section{The family of deformations}

\label{sec:Family}

We here describe the main object of study in this article: the ideals 
$J(2,P)$ which are deformations of the letterplace
ideals $L(2,P)$. The generators of $L(2,P)$ are monomials $p_1 q_2$
where $p \leq q$. We shall deform each such generator, and the
ideal $J(2,P)$ will be the ideal generated by these deformations.
We do this for the situation that the Hasse diagram of the poset $P$ is a 
rooted tree. Except for Section 5, this is our assumption throughout the paper.

The root of $P$ is at the bottom. If an element $b$ covers $a$ we
say that $a$ is a parent of $b$ and we write $a \prec b$. Two elements 
$b$ and $c$ are called siblings if they have the same parent.

For each pair $q,p$ where the meet of $q$ and $p$ is the parent of $p$,
we introduce a variable $u_{q,p}$.
Let $b$ and $c$ be distinct siblings. 
We define
\[ T_c(b) = - \sum_{q \geq c} q_2 u_{q,b}. \]
%where the $u_{p,b}$ are variables in a polynomial ring.
If $a$ is a parent of $b$ we let 
\[ T_a(b) = - a_2 u_{a,b}. \]
We also define
\[ T(b) = T_b(b) = - T_a(b) -\sum_{c} T_c(b), \]
where we sum over siblings $c$ of $b$, distinct from $b$.
If $\rot$ is the root of $P$ we define 
\[T(\rot)= u_{\emptyset,\rot}. \]

The rationale for introducing these variables will become clear
in Subsection \ref{subsec:DecoLP} where we compute
the cotangent cohomology of the ring $\kr[x_{[2] \times P}]/L(2,P)$,
Corollary \ref{cor:CotangRP}.
Let $B$ be the polynomial ring in these variables
\[ B = \kr[u_{\emptyset, \rot}, u_{q,p}] \]
ranging over all pairs $(q,p)$ such that the meet of $q$ and $p$ 
is the parent of $p$. This will be the base ring for our family of 
deformations. Let $B(2,P)$ 
be the ring $B \te_\kr \kr[x_{[2] \times P}]$. 
This is the ring where the ideal of the full deformation family lives.

\medskip
For an element $p$ in the poset $P$ we define the {\it depth} of $p$,
$\depth(p)$
to be the length of the longest chain upwards, starting from $p$.
Thus if $p$ is a maximal element, then $\depth(p) = 0$. 

We now define the following. 
\begin{itemize}
\item If $a \prec b$ so $a$ is the parent of $b$, we shall define determinants $D(a)^b$ lying in $B(2,P)$,
as well as determinants $D(a)^a$.  
\item If $a \leq b$ we shall define polynomials $S_a(b_2)$ lying in $B(2,P)$.
Then $S_a$ extends uniquely to a linear map on linear combinations of elements
$b_2$ where $b \geq a$. For short we shall often write $S_a(b)$ for 
$S_a(b_2)$.
\end{itemize}
We shall do this inductively on $\depth(a)$. 

Given these definitions, we also define the following:
\begin{itemize}
\item If $b$ and $c$ are distinct siblings let
\[S_cT_c(b) = S_c(T_c(b)). \] This definition will appear
by induction on $\depth(c)$ as we define $S_c$.
\end{itemize}

We also define the following:
\begin{itemize}
\item $S_bT_b(b) = b_1$.
\item If $a \prec b$ so $a$ is the parent of $b$,  let $S_aT_a(b) = -u_{a,b}$.
\end{itemize}
Note that these last two definitions are {\it not} compositions,
i.e. $S_bT_b(b)$ is not $S_b(T_b(b))$. Rather we 
think of these definitions as symbolic expressions.

\medskip
Now let us start the inductive definitions.
If $a$ is maximal, that is $\depth(a) = 0$ we define $D_a(a) = 1$
and $S_a(a) = 1$. 

Otherwise let $b^1, \ldots, b^m$ be the children of $a$. For
uniformity we denote $b^0 = a$. We form the $m \times (m+1)$ matrix
$M(a) = [S_{b^i}T_{b^i}(b^j)]$ where the column index $i = 0, \ldots ,m$ and 
the row index $j = 1, \ldots, m$. 
\begin{align*} & \left [ \begin{matrix} S_aT_a(b^1) & S_{b^1}T_{b^1}(b^1) & 
\cdots & S_{b^m}T_{b^m}(b^1) \\
S_aT_a(b^2) & S_{b^1}T_{b^1}(b^2) & 
\cdots & S_{b^m}T_{b^m}(b^2) \\
\vdots & \vdots &  & \vdots \\
S_aT_a(b^m) & S_{b^1}T_{b^1}(b^m) & 
\cdots & S_{b^m}T_{b^m}(b^m) 
\end{matrix} \right ] \\
 = &   \left [ \begin{matrix} -u_{a,b^1} & b^1_1 & 
\cdots & S_{b^m}T_{b^m}(b^1) \\
-u_{a,b^2} & S_{b^1}T_{b^1}(b^2) & 
\cdots & S_{b^m}T_{b^m}(b^2) \\
\vdots & \vdots &  & \vdots \\
-u_{a,b^m} & S_{b^1}T_{b^1}(b^m) & 
\cdots & b^m_1 
\end{matrix} \right ]
\end{align*}

Let $M(a)^{b^i}$ be the matrix obtained by deleting column $i$. 
Define the signed determinant
\[ D(a)^i = D(a)^{b^i} = (-1)^i |M(a)^{b^i}|.\]
Note that in order to define this, we need to have defined 
$S_p$ for all $p$ strictly bigger than $a$.

For $a \leq b$ define $R(a,b) = 1$ if $a = b$ and if $a < b$ define 
\[ R(a,b) = \prod_{a \leq p \prec q \leq b} D(p)^q \]
where the product is over all covering relations $p \prec q$ between
$a$ and $b$.

Now define 
\[ S_a(b_2) = R(a,b)D(b)^b. \]
As said before this definition extends in a natural way to linear 
combinations of variables $b_2$ such that $b\geq a$. Also 
we shall often for short write $S_a(b)$ for $S_a(b_2)$.
This completes the inductive definitions. Now we give the ideal
defining the full family of deformations of $L(2,P)$.

\medskip
\begin{definition} \label{def:FamilyST}
Let $J(2,P)$ be the ideal in $B(2,P)$ generated by 
\[ p_1 q_2 - T(p)S_p(q) \]
for all $p \leq q$.
\end{definition}

The following shows the various entities defined, in the simplest
situation.
It is useful in the next section where we give examples
of the ideals $J(2,P)$.

\begin{lemma} \label{lem:FamSingle}
Suppose $a \in P$ has a single child $b$. Then:
\begin{itemize}
\item[a.] $S_a(a) = b_1$,
\item[b.] $D(a)^a = b_1$,
\item[c.] $D(a)^b = u_{a,b}$,
\item[d.] $T(b) = a_2u_{a,b}$.
\end{itemize}
\end{lemma}

\begin{proof}
In this case the matrix $M(a) = \begin{bmatrix}
-u_{a,b} & b_1
\end{bmatrix}$.
\end{proof}

We shall in the Section \ref{sec:Flat} show that the ring $B(2,P)/J(2,P)$ 
is a deformation of 
the letterplace ring $\kr[2,P]/L(2,P)$, flat over the base ring $B$.

\section{Examples}

We consider here four examples of posets 
and give the deformed equations explicitly.
We also identify the variety they define.

\subsection{Determinantal variety}
Let $P$ be the totally ordered poset $[n] = \{ 1 < 2 < \cdots < n\}$.
For simplicity we assume $n= 4$ and write $P = \{ a < b < c < d\}$.

The deformations of  monomials $p_1 p_2$ for $p \in P$ are
\[ p_1 p_2 - T(p) S_p(p). \]
Since in this case $p$ has one or none child, and also no sibling,
we apply Lemma \ref{lem:FamSingle} and these deformations are:
\begin{align*}
& a_1a_2 - u_{\empt,a} b_1 \\
& b_1b_2 - a_2u_{a,b}c_1 \\
& c_1c_2 - b_2u_{b,c}d_1 \\
& d_1d_2 - c_2u_{c,d}
\end{align*}

Furthermore we have the deformed polynomials:
\begin{align*}
& a_1b_2 - u_{\empt,a}u_{a,b} c_1 & & a_1 c_2 - u_{\empt,a}u_{a,b}u_{b,c}d_1 & 
& a_1 d_2 - u_{\empt,a}u_{a,b}u_{b,c}u_{c,d} & \\
& b_1c_2 - a_2u_{a,b}u_{b,c} d_1 & & b_1 d_2 - a_2u_{a,b}u_{b,c}u_{c,d} & 
& & \\
& c_1d_2 - b_2u_{b,c}u_{c,d}. & & & 
& & 
\end{align*}

These binomials are the $2$-minors of the following $2 \times 5$ matrix
\[ \left [ 
\begin{matrix} a_1 & b_1 & u_{a,b}c_1 & u_{a,b}u_{b,c}d_1 & 
u_{a,b}u_{b,c}u_{c,d} \\
u_{\empt,a} & a_2 & b_2 & c_2 & d_2 
\end{matrix} \right ],  \]
after we localize by inverting $u_{a,b}$ and $u_{b,c}$.
%\comment{It does not seem that we need to invert $u_{a,b}$ and $u_{b,c}$}
This matrix can also be written as: 
\[  \left [ 
\begin{matrix} a_1 & S_a(a) & S_a(b) & S_a(c) & 
S_a(d) \\
u_{\empt,a} & a_2 & b_2 & c_2 & d_2 
\end{matrix} \right ]. \]

\subsection{The Pfaffians of $5 \times 5$ skew-symmetric matrices}
%Grassmannian $G(2,5)$ and further on }

Consider the star poset $P$ 
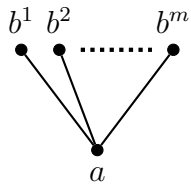
\begin{center}
\begin{tikzpicture}[scale=1, vertices/.style={draw,fill=black, circle, inner 
sep=1.5pt}]
\node [vertices, label=below:{$a$}] (0) at (0,0){};
\node [vertices, label=above:{$b^1$}] (1) at (-1,1.33){};
\node [vertices, label=above:{$b^2$}] (2) at (-.5,1.33){};
\node [vertices, label=above:{$b^m$}] (3) at (+1,1.33){};
\node [vertices, white] (4) at (-.3,1.33){};
\node [vertices, white] (5) at (.8,1.33){};
\foreach \to/\from in {0/1, 0/2, 0/3}
\draw [-,thick] (\to)--(\from);
\foreach \to/\from in {4/5}
\draw [dotted,ultra thick] (\to) -- (\from);
\end{tikzpicture}
\end{center}

We will show the deformations
when the root has two or three children.
First suppose we have two children so the poset is

%\begin{center}
%\begin{tikzpicture}[scale=.5, vertices/.style={draw, fill=black, circle, inner sep=1pt}]
%              \node [vertices, label=right:{$a$}] (0) at (-0+0,-1){};
%              \node [vertices, label=right:{$b$}] (1) at (-0.8,0.5){};
%              \node [vertices, label=right:{$c$}] (2) at (+0.8,0.5){};
%      \foreach \to/\from in {0/1, 0/2}
%      \draw [-] (\to)--(\from);
%\end{tikzpicture}
%\end{center}

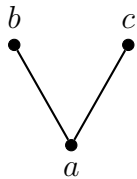
\begin{center}
\begin{tikzpicture}[scale=1, vertices/.style={draw,fill=black, circle, inner 
sep=1.5pt}]
\node [vertices, label=below:{$a$}] (0) at (0,0){};
\node [vertices, label=above:{$b$}] (1) at (-.75+0,1.33){};
\node [vertices, label=above:{$c$}] (2) at (-.75+1.5,1.33){};
\foreach \to/\from in {0/1, 0/2}
\draw [-,thick] (\to)--(\from);
\end{tikzpicture}
\end{center}

The matrix $M(a)$ is:
\[ \left [ 
\begin{matrix}
S_aT_a(b) & S_bT_b(b) & S_cT_c(b) \\
S_aT_a(c ) & S_bT_b(c) & S_cT_c(c) 
\end{matrix}  \right ]
= \left [
\begin{matrix}
-u_{a,b} & b_1 & -u_{c,b} \\
-u_{a,c} & -u_{b,c} & c_1
\end{matrix}
\right ].
\]
There are five generating monomials in $L(2,P)$ and the deformed
polynomials are:
\begin{align*}
& b_1 b_2 + T(b)\cdot 1 & \\
& c_1 c_2 + T(c) \cdot 1 & \\
& a_1b_2 + T(a)S_a(b) & \\
& a_1c_2 + T(a)S_a(c) & \\
& a_1a_2 + T(a)S_a(a)
\end{align*}
which are:
\begin{align*}
& b_1 b_2 - a_2u_{a,b} - c_2u_{c,b} & \\
& c_1 c_2 - a_2u_{a,c} - b_2u_{b,c} & \\
& a_1b_2 - u_{\empt,a}u_{a,c}u_{c,b} - u_{\empt,a}u_{a,b}c_1 & \\
& a_1c_2 - u_{\empt,a}u_{a,b}u_{b,c} - u_{\empt,a}u_{a,c}b_1 & \\
& a_1a_2 - u_{\empt,a}b_1c_1 + u_{\empt,a}u_{c,b}u_{b,c}
\end{align*}
These are (after diving by $u_{\empt,a}$, 
the Pfaffians of the following skew-symmetric matrix:
\[ \left [ \begin{matrix} 
0 & u_{b,c} & c_2 & b_1 & a_1 \\
-u_{b,c} & 0 & u_{a,c} & u_{\empt,a}^{-1}a_1 & c_1 \\
-c_2 & -u_{a,c} & 0 & u_{a,b} & b_2 \\
-b_1 & -u_{\empt,a}^{-1}a_1 & -u_{a,b} & 0 & u_{c,b} \\
-a_2 & -c_1 & -b_2 & -u_{c,b} & 0
\end{matrix} \right ].
\]
Setting $u_{\empt, a} = 1$, they are also the Pl\"ucker relations
defining the Grassmann variety $G(2,5)$ embedded in projective
space ${\mathbb P}^9$. 
%These Pl\"ucker relations can also be obtained from the above
%by the variable change $a_1 \mapsto u_{\empt,a}a_1$, and then inverting
%$u_{\empt,a}$. 

\medskip
Now let us consider the case of the star poset $P$ with three childen:
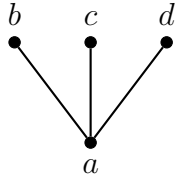
\begin{center}
\begin{tikzpicture}[scale=1, vertices/.style={draw,fill=black, circle, inner 
sep=1.5pt}]
\node [vertices, label=below:{$a$}] (0) at (0,0){};
\node [vertices, label=above:{$b$}] (1) at (-1,1.33){};
\node [vertices, label=above:{$c$}] (2) at (0,1.33){};
\node [vertices, label=above:{$d$}] (3) at (+1,1.33){};
%\node [vertices, white] (4) at (-.3,1.33){};
%\node [vertices, white] (5) at (.8,1.33){};
\foreach \to/\from in {0/1, 0/2, 0/3}
\draw [-,thick] (\to)--(\from);
%\foreach \to/\from in {4/5}
%\draw [dotted,ultra thick] (\to) -- (\from);
\end{tikzpicture}
\end{center}
The matrix $M(a)$ is:
\[ \left [ \begin{matrix}
-u_{a,b} & b_1 & -u_{c,b} & -u_{d,b} \\
-u_{a,c} & -u_{b,c} & c_1 & -u_{d,c} \\
-u_{a,d} & -u_{b,d} & -u_{c,d} & d_1
\end{matrix} \right ], \]
and the $D(a)^x$ are the signed maximal minors of this matrix.
There are seven generating monomials in $L(2,P)$. Their deformations
are the following:
\begin{align} \notag
& b_1b_2 - a_2u_{a,b} - c_2u_{c,b} - d_2u_{d,b} & \\ \notag
& c_1c_2 - a_2u_{a,c} - b_2u_{b,c} - d_2u_{d,c} & \\ \notag
& d_1d_2 - a_2u_{a,d} - b_2u_{b,d} - c_2u_{c,d} & \\ \label{eq:ExStarForms}
& a_1b_2 - u_{\empt,a}D(a)^b & \\ \notag
& a_1c_2 - u_{\empt,a}D(a)^c & \\ \notag
& a_1d_2 - u_{\empt,a}D(a)^d & \\ \notag
& a_1a_2 - u_{\empt,a}D(a)^a &  
\end{align}

\begin{question}
Is there a natural description of the variety defined by these equations?
\end{question}

\subsection{Ladder determinantal varieties}
Consider now the poset $P$:

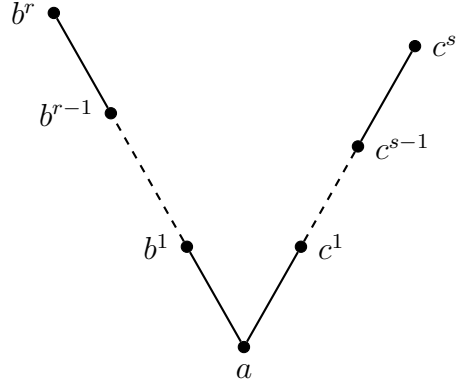
\begin{center}
\begin{tikzpicture}[scale=1, vertices/.style={draw,fill=black, circle, inner 
sep=1.5pt}]
\node [vertices, label=below:{$a$}] (0) at (0,0){};
\node [vertices, label=left:{$b^1$}] (1) at (-.75+0,1.33){};
\node [vertices, label=right:{$c^1$}] (2) at (-.75+1.5,1.33){};
\node [vertices, label=right:{$c^{s-1}$}] (3) at (1.5,2.66){};
\node [vertices, label=right:{$c^s$}] (4) at (1.5+.75,3.99){};
\node [vertices, label=left:{$b^{r-1}$}] (5) at (-1-.75,1.77+1.33){};
\node [vertices, label=left:{$b^r$}] (6) at (-1-1.5,1.77+2.66){};
\foreach \to/\from in {0/1, 0/2, 3/4, 5/6}
\draw [-,thick] (\to)--(\from);
\foreach \to/\from in {2/3, 1/5}
\draw [dashed,thick] (\to) -- (\from);
\end{tikzpicture}
\end{center}
where $r,s \geq 2$.
We let $b^{r+1}_1 = 1 = c^{s+1}_1$, and for short write $b = b^1$ and $c = c^1$.
The matrix $M(a)$ is:
\[ \left [ \begin{matrix}
-u_{a,b} & b_1 & S_cT_c(b) \\
-u_{a,c} & S_bT_b(c) & c_1
\end{matrix} \right ].  \]
The monomials in $L(2,P)$ deform to the following polynomials 
generating $J(2,P)$:
\begin{align}
&  b^i_1 b^j_2 - b^{i-1}_2 (\prod_{k=i}^j u_{b^{k-1},b^k}) b^{j+1}_1 & 
2 \leq i \leq j \leq r \label{eq:Exbb}\\
& c^i_1 c^j_2 - c^{i-1}_2 (\prod_{k = i}^j u_{b^{k-1},b^k}) c^{j+1}_1  &
2 \leq i \leq j \leq s \notag \\
& b^1_1b^j_2 - (a_2u_{a,b^1} + \sum_{j = 1}^s c^j_2u_{c^j,b^1})
(\prod_{k = 2}^j u_{b^{k-1},b^k}) b^{j+1}_1 & 1 \leq j \leq r \label{eq:Exb1b}\\
& c^1_1c^j_2 - (a_2u_{a,c^1} + \sum_{j = 1}^s b^j_2u_{b^j,c^1})
(\prod_{k = 2}^j u_{c^{k-1},c^k}) c^{j+1}_1 & 1 \leq j \leq r \notag \\
& a_1 b^j_2 - u_{\empt,a}D(a)^b (\prod_{k = 2}^j u_{b^{k-1},b^k}) b^{j+1}_1 & 
1 \leq j \leq r \label{eq:Exab}\\
& a_1 c^j_2 - u_{\empt,a}D(a)^c (\prod_{k = 2}^j u_{c^{k-1},c^k}) c^{j+1}_1 & 
1 \leq j \leq s \notag \\
& a_1a_2 - u_{\empt,a} D(a)^a \label{eq:Exaa}
\end{align}

Write $T = a_2u_{a,b}u_{a,c} + u_{a,b}T_b(c) + u_{a,c}T_c(b)$. 
We claim that the ideal $J(2,P)$ they generate is precisely
the ideal of two-minors of the ladder:
\[ \begin{matrix}
\vdots & \vdots & \vdots & &   \\
S_c(c^2) & c^2_2 & 0 & 0 & \\
S_c(c^1) & c^1_2 & 0 & 0 & \cdots \\
D(a)^b & T & b^1_2 & b^2_2 & \cdots\\
u_{\empt,a}^{-1} a_1 & D(a)^c & S_b(b^1) & S_b(b^2) & \cdots 
\end{matrix}, \]
if we localize by inverting $u_{\empt,a}, u_{a,b}$ and $u_{a,c}$.
Let us call the part of the above ladder starting
from the column with $b^1_2$, the left leg.
The minors of the left leg are precisely the equations \eqref{eq:Exbb}.
The minors formed by taking the first column (only the two lowest
entries) and a column of the left leg, gives the equations \eqref{eq:Exab}.
Now by multiplying the column of $b^i_2$ with $u_{a,b}u_{b^i,c}$ 
and subtracting the sum of all these scaled columns from the column
with $T$, we obtain a column
\[ \left [ \begin{matrix} u_{a,c}T_c(b) + a_2u_{a,b}u_{a,c} \\
- u_{a,c} b_1
\end{matrix} \right ].
\]
Taking the determinant of this column and the columns of the left leg,
we obtain the equations \eqref{eq:Exb1b} after inverting $u_{a,c}$.

Lastly we want to obtain the equation \eqref{eq:Exaa}.
Take the determinant of the lower left $2 \times 2$ matrix in the ladder.
This is (after mulitplying with $u_{\empt,a}$):
\begin{align*} & a_1 T - u_{\empt,a}D(a)^bD(a)^c &   \\
= & u_{a,b}u_{a,c}a_1a_2 + a_1u_{a,b}T_b(c) + a_1u_{a,c}T_c(b)& \\
+ & u_{\empt,a}[ u_{a,b}u_{a,c}b_1 c_1 - c_1u_{a,b}^2 S_bT_b(c) - b_1u_{a,c}^2
S_cT_c(b) + u_{a,b}u_{a,c} S_bT_b(c)S_cT_c(b)].&
\end{align*}
By subtracting the following polynomial, obtained as a linear
combination  of the $r$ first equations in \eqref{eq:Exab}:
\[ u_{a,b}[ a_1 T_b(c) - u_{\empt,a}D(a)^b S_bT_b(c)]  \]
and the polynomial, obtained as a linear combination of the
$s$ second equations in \eqref{eq:Exab}:
\[ u_{a,c}[ a_1 T_c(b) - u_{\empt,a}D(a)^c S_cT_c(b)],  \] 
we obtain the last equation \eqref{eq:Exaa}.

\subsection{Variety of complexes}
For definition of variety of complexes refer to \cite{DeConcini198157}.
Let $P$ be the poset with Hasse diagram:

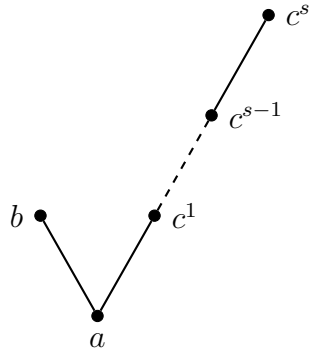
\begin{center}
\begin{tikzpicture}[scale=1, vertices/.style={draw,fill=black, circle, inner sep=1.5pt}]
\node [vertices, label=below:{$a$}] (0) at (0,0){};
\node [vertices, label=left:{$b$}] (1) at (-.75+0,1.33){};
\node [vertices, label=right:{$c^1$}] (2) at (-.75+1.5,1.33){};
\node [vertices, label=right:{$c^{s-1}$}] (3) at (1.5,2.66){};
\node [vertices, label=right:{$c^s$}] (4) at (1.5+.75,3.99){};
\foreach \to/\from in {0/1, 0/2, 3/4}
\draw [-,thick] (\to)--(\from);
\draw [dashed,thick] (2) -- (3);
\end{tikzpicture}
\end{center}
where $s\geq 2$. We write $c = c^1$ and define $c^{s+1}_1 = 1$.
We have
\[
M(a) = \begin{bmatrix}
-u_{a,b} & b_1 & -\sum_{j=1}^s u_{c^j,b} c_1^{j+1} \prod_{k=2}^j u_{c^{k-1},c^k}\\
-u_{a,c} & -u_{b,c} & c_1\\
\end{bmatrix}
\]
The deformed relations of $L(2,P)$ are the following:
\begin{align}
\label{VC01} & c^i_1c^j_2 - c^{i-1}_2 (\prod_{k=i}^j u_{c^{k-1},c^k}) c^{j+1}_1 & 2\leq i\leq j\leq s\\
\label{VC02} & c^1_1c^j_2 - (a_2u_{a,c} + b_2 u_{b,c}) \prod_{k=2}^j u_{c^{k-1},c^k} c^{j+1}_1 & 1\leq j\leq s\\
\label{VC03} & b_1b_2 - a_2u_{a,b} - \sum_{j=1}^s c^j_2u_{c^j,b}\\
\label{VC04} & a_1c^j_2 - u_{\emptyset,a} D(a)^{c} (\prod_{k=2}^j u_{c^{k-1},c^k}) c^{j+1}_1 & 1\leq j\leq s\\
\label{VC05} & a_1b_2 - u_{\emptyset,a} D(a)^b\\
\label{VC06} & a_1a_2 - u_{\emptyset,a} D(a)^a
\end{align}

The 2-minors of the matrix
\[ X=
\begin{bmatrix}
a_1 & c^1_1 u_{\emptyset,a} & c^2_1 u_{\emptyset,a} u_{a,c^1} & \cdots & c^{s+1}_1 u_{\emptyset,a} u_{a,c^1} \prod_{k=2}^s u_{c^{k-1},c^k}\\
b_1 + u_{a,b} u_{a,c}^{-1} u_{b,c} & a_2 + b_2 u_{a,c}^{-1} u_{b,c} &  c^1_2 & \cdots & c^s_2\\
\end{bmatrix}
\]
except the one associated to the first two columns give us the relations \eqref{VC01}, \eqref{VC02} and \eqref{VC04}.
Now define the matrix $Y$ as
\[
\begin{bmatrix}
b_2 & -u_{a,b} & -u_{c^1,b} & \cdots & -u_{c^s,b}
\end{bmatrix}
\]
The two entries of the matrix $XY$ give the relations \eqref{VC03} and \eqref{VC05}. If we multiply equation \eqref{VC05} with $u_{a,c}^{-1}u_{b,c}$ and subtract
it from the minor of the first two columns of $X$, we get equation \eqref{VC06}.

\medskip
In the appendix we give a larger poset $P$ and  the generators 
of $J(2,P)$. As we see these polynomials grow quickly in size.

\section{The fine positive grading on the ideal $J(2,P)$}
\label{sec:Grad}

In this section we show that the ideal $J(2,P)$ is very finely
graded. In fact it is graded by a free abelian group on
$2|P|$ free generators. We show that this grading is positive
in the sense of \cite{HaSt}. This enables us to state
a simple criterion for flatness of homogeneous ideals, which
we will apply in Section \ref{sec:Flat} to conclude that the
quotient ring $B(2,P)/J(2,P)$ is flat over its base ring 
$B = \kr[u_{\empt,\rot}, u_{q,p}]$.

\subsection{The grading on $J(2,P)$}
Let $\ztop$ be the free abelian group of order $2|P|$ generated by  
the $p_1$'s and $p_2$'s for $p \in P$. The ideal $J(2,P)$ lives
in the polynomial ring 
\[ B(2,P) =  \kr[x_{[2] \times P}]  \te_{\kr} 
\kr[u_{\empt,\rot}, u_{q,p}]. \]
The pairs $(q,p)$ are all pairs such that the meet of $q$ and $p$
is the parent of $p$. They, together with $u_{\empt,\rot}$
correspond to the minimal 
generators of the first cotangent module $T^1(\kr[x_{[2] \times P}]/L(2,P))$
by Corollary \ref{cor:CotangRP}.
For an element $p \in P$ let $b^1, \ldots, b^m$ be its children.
Denote by $\hat{p}$ the degree $p_2 - b^1_1 - b^2_1 - \cdots - b^m_1$ in 
$\ztop$.

\medskip
Now define a grading on the $B(2,P)$ by letting the variable
$x_{i,p}$ (which we write as $p_i$)
have degree (with some abuse of notation) $p_i$ in the abelian group $\ztop$.
Also define the degree of $u_{q,p}$ to be $
p_1 -q_2 + \hat{p}$.

%\begin{proposition}
%The ideal $J(2,P)$ is homogeneous for this $\ztop$-grading.
%Moreover:
%\begin{itemize}
%\item $T(p)$ is homogeneous of degree $p_1 + p_2 - \hat{p}$
%\item Let $b$ be a child of $a$. The determinant $D_b(a)$ is homogeneous 
%of  degree $\hat{a} - \hat{b} + b_2 - a_2$. The determinant $D_a(a)$
%is homogeneous of degree $\hat{a}$. 
%\item $S_qT_q(p)$ is homogeneous of degree 
%$p_1 + p_2 - \hat{p} - q_2 + \hat{q}$. 
%\end{itemize}
%\end{proposition}

%Can you check if this holds, and if so, see if you can provide an 
%argument for this. 
%(I don't think this should be very difficult.)

\begin{proposition}
\label{pro:Graded}
The ideal $J(2,P)$ is homogeneous for this $\mathbb{Z}([2]\times P)$-grading. Moreover:
\begin{enumerate}
\item $T(p)$ is homogeneous of degree $p_1+ \hat{p}$.
\item If $p \leq q$ then $S_p(q)$ is homogeneous of degree
$q_2 - \hat{p}$.
\item When $p$ and $q$ are siblings, $S_qT_q(p)$ 
is homogeneous of degree $p_1+\hat{p}-\hat{q}$.
\item Let $q$ be a child of $p$. The determinant 
$D(p)^q$ is homogeneous of degree $\hat{q}- \hat{p}$. 
The determinant $D(p)^p$ is homogeneous of degree $p_2 - \hat{p}$.
\end{enumerate}
\end{proposition}

\begin{proof}
(1) easily follows from the definitions. Note that if $\rho$ is the root then 
$\deg(T(\rho)) = \deg(u_{\emptyset,\rho}) = \rho_1 + \hat{\rho}$.

We prove (2) and (4) simultaneously since they are both dependent on each other. (3) follows from (2) by definition.
Let $m$ be the cardinality of $P$ and $\eta : P \rarr [m]$ a linear
extension.

%Let $A$ be the set of elements of poset $P$ with a total order such that parent of an element $p$ is always bigger than $p$. Note that such a total order exists.
%(One way to do this is to choose a maximal element $p^1$ of poset $P$ to be the smallest element of $A$ and then remove this element from $P$. Then we choose a maximal element in the new poset to be the point $p^2$ in $A$ that covers $p^1$ and so on.)

We prove (2) and (4) by descending induction on $\eta(p)$.
Suppose $\eta(p) = m$, the maximal possible.
Then $D(p)^{p} = S_{p}(p) = 1$ has degree $p_2 - \widehat{p} = 0$.
 
Now we show that if for any two elements $p \leq q$ such that $\eta(p) = k+1$
we have 
\[
\deg D(p)^p = p_2 - \hat{p}, \quad \deg D(p)^q = \hat{q}-\hat{p}, 
\quad \deg S_p(q) = q_2 - \hat{p} \]
then the above statement is also true for any two elements $p \leq q$
with $\eta(p) = k$. 

So suppose $\eta(p) = k$ and let $b^1,\ldots,b^m$ be children of $p$. These
have all $\eta$-value 
$\geq k+1$. By assumption we have
%\comment{I changed notation, so fix the gradings here}
\begin{align*}
\deg(D(p)^p) & = \deg \big( \sum_{\sigma \in S_m} \prod_{i=1}^m S_{b^i} T_{b^i} (b^{\sigma(i)})\big)\\
& = \sum_{i=1}^{m} b^i_1 + \sum_{i=1}^{m} \widehat{b^i} - \sum_{i=1}^{m} \widehat{b^i} = p_2 - \hat{p}. 
\end{align*}
where $S_m$ is symmetric group on $m$ letters.

Similarly when $q$ is a child of $p$, we have $\deg(D(p)^q) = \hat{q}-\hat{p}$.
For $p \leq q$ in $P$ let $p=x^0,\ldots,x^{m}=q$ be the maximal chain between 
$p$ and $q$ in $P$. Then
\begin{align*}
\deg (S_p(q)) & = \deg\big( D(p)^{x^1} D(x^1)^{x^2} \cdots D(x^{m-1})^{q}D(q)^q \big)\\
& = (\widehat{x^1}- \hat{p}) + (\widehat{x^2} - \widehat{x^1}) + \cdots+ (\hat{q} - \hat{x^{m-1}}) + (q_2 - \hat{q})\\
& = q_2-\hat{p},
\end{align*}
This completes the proof of (2)-(4). 
The main assertion now follows from (1) and (2).
\end{proof}

\subsection{Positive gradings by an abelian group}

Let $Y$ be a finite-dimensional vector space graded by an abelian
group $A$. This gives an $A$-graded polynomial ring $\kr[Y]$.
This grading is {\it positive} if the only elements in $\kr[Y]$ of 
degree $0$ are the constants in $\kr$. 
By \cite[Prop. 8.6]{StMi}, this is equivalent to each
graded piece $\kr[Y]_{a}, a \in A$ being a finite-dimensional
vector space over $\kr$. The following is another characterization.

\begin{lemma}
The grading by $A$ on $\kr[Y]$ is positive, iff there is a homomorphism 
$A \rarr \ZZ$ such that the $\ZZ$-degree of any $A$-homogeneous 
element of $Y$ is positive, i.e. is in  $\ZZ_{> 0}$.
\end{lemma}

\begin{proof} The if direction is clear. The only if direction
follows by \cite[Cor. 8.7]{StMi} by applying this to the
torsion free part of $A$, which must be nonzero.
(Note that the definition of {\it positive} grading in \cite[Chap.8]{StMi}
requires $A$ to be torsion-free. We follow the convention of
\cite{HaSt} which does not have this requirement.)

\end{proof}

\begin{proposition}
The grading on $B(2,P)$ given in Proposition \ref{pro:Graded}
is positive.
\end{proposition}

\begin{proof}
Define a map from $\ZZ([2] \times P) \mto{d} \ZZ$ by letting every  
$d(p_2) = 1$ and define $d(p_1)$ inductively on $\depth(p)$ such
that the $d(p)$ is positive and larger than the sum
$\sum_{b}d(b)$ where we sum over the children $b$ of $p$.
Then it is seen that all variables $p_i$ have positive $d$-values,
as well as all variables $u_{q,p}$. 
\end{proof} 

Let $J \sus \kr[Y]$ be an ideal which is homogeneous for a {\it positive} 
$A$-grading on $Y$.  Let $U \sus Y$ be a homogeneous subspace for this
grading, and
%For now we assume that $J$ is {\it not} necessarily flat over $\kr[U]$.
$I \sus \kr[Y/U]$ be the ideal such that 
\[ \kr[Y/U]/I = (\kr[Y]/J) \te_{\kr[Y]} \kr[Y/U]. \]
Let $f_1, \ldots, f_k$ be generators of $J$ and let $\ov{f_1}, \ldots,
\ov{f_k}$ denote their images in $I$, which will generate $I$.
The following is the criterion we use, in Theorem \ref{thm:Flat},
to show that $J(2,P)$ is a flat deformation of $L(2,P)$.

\begin{proposition} \label{pro:GradedFlatlift}
Suppose the $A$-grading on $\kr[Y]$ is positive.
If every relation of $\ov{f_1}, \ldots, \ov{f_k}$ lifts to a relation
for $f_1, \ldots, f_k$, then $\kr[Y]/J$ is flat over $\kr[U]$.
\end{proposition}

\begin{proof}
We apply the local criterion of flatness \cite[Thm.A.5]{Ser}
and the criterion of lifting of relations \cite[Thm.A.10]{Ser},
to the algebra homomorphism
$\kr[Y]/J \rarr \kr[Y/U]/I$.
The situation in \cite[Thm.A.5]{Ser} 
is a local homomorphism of local noetherian rings.
Our situation is a graded homomorphism of positively graded rings.
This situation works just as well for the arguments.
\end{proof}

%\amincomment{Let $A$ be an abelian group. By Proposition \ref{I2PGraded}, any $A$-grading on $\kr[x_{[2]\times P}]$ induces a grading on $B(2,P)$ such that $J(2,P)$ becomes a homogeneous $A$-graded ideal.}

%\medskip
%\amincomment{
%Now let $m$ be a monomial of degree zero in $B(2,P)$. Obviously, $m$ can not only contain variables from $\kr[x_{[2]\times P}]$. Suppose for some $p,q$ in $P$, $u_{p,q}|m$.
%Since $\deg(u_{p,q}) = q_1 - p_2 + \hat{q}$, to eliminate degree $q_1$, $m$ is a multiple of some variable $u_{p^1,q^1}$ such that $q^1$ is the parent of $q$. Again to eliminate degree $q^1_1$, $m$ contains some $u_{p^2,q^2}$ with $q^2$ the parent of $q^1$. By repeating this argument we conclude that some power of $u_{\emptyset,rt}$ divides $m$. 
%But $\deg(u_{\emptyset,rt}) = rt_1+\hat{rt}$ and the degree $rt_1$ can not be eliminated by any variable. This shows that $m$ can only be an element of $\kr$. Hence the grading is positive.
%}

\section{First order deformations and the cotangent cohomology}

This sections presents the basic general deformation theory that we need.
General references for deformation theory are 
\cite[Chap.3]{Ser}, \cite[Chap.1,3]{DefoHar} or \cite[Sec.3]{DefoStev}.
We compute explicitly the first cotangent cohomology 
$T^1(\kr[x_{[2] \times P}/L(2,P))$ for any quadratic letterplace ideal of 
a poset $P$. As we shall see the elements of this 
module are remarkably simple in form.

\subsection{The deformation functor}
\label{FamSubsecDefo}

%Let $\kartalg$ be the category of local Artinian $\kr$-algebras.
We consider a $\kr$-algebra $R$  and an ideal $I$ in $R$. 
Let $B$ be another $\kr$-algebra with a distinguished $\kr$-point 
$b \in \Spec B$ corresponing to a morphism $B \rarr \kr$. 
A deformation over $B$
of the ideal $I \sus R$,  is an ideal $J \sus R \te_\kr B$ such that: 
\begin{itemize} 
\item $(R\te_\kr B)/J$ is flat over $B$,
\item The natural map $R \te_\kr B \rarr R$ induces a map 
$(R \te_\kr B)/J \rarr R/I$ which becomes an isomorphism
\[ (R \te_\kr B)/J \te_B \kr \overset{\iso}\rarr R/I. \]
\end{itemize}
Let $\sets$ be the category of sets, and 
$\kartalg$ be the category of local artinian $\kr$-algebras 
with residue field $\kr$.

The functor
of infinitesimal deformations of $I$
\[ \Def{I} :  \kartalg \rarr  \sets, \]
is given by letting 
$\Def{I}(A)$ be the set of such deformations $J$ over the ring $A$.

\subsection{The tangent space}
Let $\kr[\epsilon] = \kr[t]/(t^2)$. The deformations over
this ring are called first order deformations.
The tangent space of the deformation functor is 
$\Def{I}(\kr[\epsilon])$ (see below).
This space identifies as $\Hom_R(I,R/I)$. 
%More precisely we show the
%following categorical relationship.
Let $\kvect$ be the category of finite-dimensional vector spaces.
For a vector space $V$ denote by $V^*$ its dual space.
There is a functor 
\[  \kvect \rarr \kartalg \]
sending $V$ to $\kV$. By abuse of notation we get a restricted
functor $\Def{I}$ from $\kvect$ where $\Def{I}(V) = \Def{I}(\kV)$.
We also have a functor
\[ \Hom_{\kr}(-, \Hom_R(I,R/I)) : \kvect \rarr \sets, \]
sending $V$ to the set of linear maps $V \rarr \Hom_R(I,R/I)$. 
The following is standard, see \cite{Ser} Proposition 3.2.1 and Definition 3.2.3.

\begin{proposition} \label{pro:DefoTangent}
There is a natural isomorphism of functors between the two functors
\[ \Def{I}, \, \, \Hom_{\kr}(-, \Hom_R(I,R/I)) : \kvect \rarr \sets .\]
%Let $V$ be a finite-dimensional vector space, and $V^*$ its dual.
%There is a one-to-one
%orrespondence between:
%\begin{itemize}
%\item Linear maps $V^* \rarr \Hom_R(I, R/I)$
%\item Elements of $\Def{I}(\kr[V]/(V^2))$, i.e. flat deformations of $I$ over
%$\kr[V]/(V^2)$.
%\end{itemize}
%In particular the tangent space of $\Def{I}(A)$ identifies as 
%$\Hom_R(I,R/I)$.
\end{proposition}

The upshot is that the tangent space $\Def{I}(\kr[\epsilon])$ 
of the functor $\Def{I}$ on
$\kartalg$ identies as $\Hom_R(I,R/I))$.
We describe the isomorphism of functors in  more detail.

\medskip
\noindent 1. Choose a basis $v_1, \ldots, v_r$ for $V$. We obtain a dual basis
$v_1^*, \ldots, v_r^*$ for $V^*$.
A linear map $\phi : V  \rarr \Hom_R(I, R/I)$ gives 
an ideal
\[ J = \{ f + \sum_i v_i^* \phi(v_i)(f) \, | \, f \in I \}. \]
This is the flat deformation corresponding to $\phi$.

Alternatively $\phi$ gives an element of $V^* \te_\kr \Hom_R(I,R/I)$
and so a map
\[ \ov{\phi} : I \rarr I \te_\kr V^* \te_\kr \Hom_R(I, R/I) \rarr V^* \te_\kr R/I. \]
Then $J$ is the ideal
\[ J = \{ f + \ov{\phi}(f) \, | \, f \in I \}. \]

\noindent 
2. Conversely given a deformation $J \sus  R \te_\kr \kV$
of $I \sus R$, flat over $\kV$. Note first that 
the inclusion $\kr \rarr \kr[V^*]/(V^*)^2$ induces $R \rarr R \te_{\kr} \kV$,
and so the ideal $I$ embeds as a linear subspace of the right ring.
There is a short exact sequence
\[ 0 \rarr R \te_\kr V^* \rarr R \te_{\kr} \kr[V^*]/(V^*)^2 \rarr R \rarr 0. \]
Tensoring this with $ -  \te_{R \te \kr[V^*]/(V^*)^2} J$ we obtain,
recall that $J$ is flat, a short exact sequence
\[ 0 \rarr V^* \te_{\kV} J  \rarr J  \mto{\pi} J / (V^* \cdot J) = I \rarr 0.\]
The term $V^*  \te_{\kV} J$ 
identifies as $V^* \cdot J = V^* \te_{\kr} I$.
An $f\in I$ lifts by $\pi$ to an $\hat{f} = f + p$ where $p \in R \te_\kr V^*$. 
Any two liftings differ by an element of $V^* \te_\kr I$, and so we
get a well-defined $R$-module map $I \rarr V^* \te_{\kr} (R/I)$ sending
\[ f \mapsto p \in V^* \te_\kr (R/I). \] But such an $R$-module map
corresponds to a linear map $V \rarr \Hom_R(I, R/I)$. 

\medskip
\medskip
Now consider the situation of a subset $T$ of $\Hom_\kr(V^*, R) = V \te_\kr R$. 
It gives a map
$V^* \rarr R \te_\kr T^*$. This gives a map of algebras:
\[ \tau : R \te_\kr \kV \rarr R \te_\kr R \te_\kr \kr[T^*]/(T^{*})^2
\rarr R \te_\kr \kr[T^*]/(T^{*})^2, \]
where in the last map we have used the multiplication on $R$.
%Let $J \sus R \te_\kr \kV$ be a flat deformation of $I$ over
%$\kV$. 

\begin{lemma} \label{lem:DecoTtoR} Let $J \sus R \te_\kr \kV$ be a 
flat deformation of $I \sus R$ over $\kV$. 
Then the image $J^\prime = \tau(J)$ is a flat deformation of $I$
over $\kr[T^*]/(T^{*})^2$. If the deformation $J$ corresponds 
to the linear map $V \rarr \Hom_R(I,R/I)$, then the deformation $J^\prime$ 
corresponds to 
the linear map $T \rarr V \te_\kr R \rarr \Hom_R(I, R/I)$, where
to define the latter map we have used the $R$-module structure on 
$\Hom_R(I,R/I)$. 
\end{lemma} 

\begin{proof}
Let the $v_j$'s form a basis for $V$ and the $t_i$'s form
a basis for $T$. If the map $T \rarr V \te_{\kr} R$
is given by 
\[ t_i \mapsto \sum_j v_j \te_{\kr} r_{ij}, \]
then the map $V^* \rarr T^*\te_{\kr} R$ is given by 
\[ v_j^* \mapsto \sum_i t_i^* \te_{\kr} r_{ij}. \]
Let the  deformation $J$ be given by 
\[ f + \sum_j v_j^* f_j, \quad f \in I. \]
The linear map $V^* \rarr T^* \te_{\kr} R$ gives the
deformation
\[ f + \sum_j \sum_i (t_i^* \te_{\kr} r_{ij}) f_j, \quad f \in I \] 
and the ideal $J^\prime$ consists of 
\[ f + \sum_j \sum_i t_i^* r_{ij}f_j, \quad f \in I \] 
But this is the deformation corresponding to the linear map that sends each $t_i$ to the $R$-module map 
\begin{align*}  \sum_j r_{ij} v_j \, : \, \, &  I  \rarr  R/I & \\
& f  \mapsto  \sum_j r_{ij} f_j, & 
\end{align*}
and so is a flat deformation.
\end{proof}

\subsection{Deformations over polynomial rings}
Now assume the ring $R$ is a polynomial ring $\kr[X]$ where
$X$ is a finite-dimensional vector space.

A first order deformation of $I$ over $\kr[\epsilon]$ is {\it trivial} 
if it is the image of 
\[ I \te_\kr \kr[\epsilon] \sus \kr[X] \te_\kr \kr[\epsilon] \]
by an (infinitesimal) coordinate change in $\kr[X] \te_\kr \kr[\epsilon]$, meaning a linear map
\[ X \rarr \kr[X] \te_\kr \kr[\epsilon] \]
lifting the canonical inclusion $X \rarr \kr[X]$.
%\comment{I wonder if it is correct to say lifting this canonical inclusion, since if we are lifting this then each $x\in X$ is mapped to itself and we are not changing coordinates?}
This linear map induces the infinitesimal coordinate change, 
an algebra homomorphism
\[ \kr[X] \te_\kr \kr[\epsilon] \rarr \kr[X] \te_\kr \kr[\epsilon]. \]
%\[ \langle x_1, \ldots, x_n \rangle \te_\kr \kr[\epsilon]
%\larrow \langle x_1, \ldots, x_n \rangle \te_\kr \kr[\epsilon].\]

\begin{example}
Let $\kr[X] = \kr[x,y]$ and $I = (xy)$. The map 
\[ x \mapsto x + \epsilon y^3, \quad y \mapsto y + \epsilon xy \]
is an infinitesimal coordinate change. The ideal $I$ is
mapped to the ideal $(xy + \epsilon x^2y + \epsilon y^4)$ by
this coordinate change.
\end{example}

Let $\Der_\kr (\kr[X])$ be the derivations of $\kr[X]$. 
If the $x_i$ form a basis for $X$, 
this is a free $\kr[X]$-module generated by the derivatives
$\partial/\partial x_i$ for $i = 1, \ldots, n$.
There is a map
\begin{align} \label{eq:DecoDelta}
\delta^\ast :\Der_K(\kr[X]) & \to \Hom_{\kr[X]}(I,\kr[X]/I)
%\partial &\mapsto \delta^\ast(\partial)(f_i) =  \partial f_i + I
\end{align}
which sends $\partial$ to the homomorphism sending $f_i$ to $\partial f_i + I$.
The image of the homomorphism $\delta^\ast$ consists of all the trivial 
infinitesimal deformations of $\kr[X]/I$.

\medskip
There is a diagonal
map $X \rarr X \oplus X$ inducing
\[ \tau : \kr[X] \rarr \kr[X \oplus X] = \kr[X] \te_\kr \kr[X] \rarr
\kr[X] \te_\kr \kr[X]/(X)^2. \]

\begin{example}
Let $X$ have basis $x_1, x_2$ and let the other copy of $X$
have basis $t_1, t_2$. The image by $\tau$ 
of the ideal $(x_1^2x_2)$ in $\kr[x_1,x_2]$
is then the ideal generated by 
\[ (x_1 + t_1)^2(x_2 + t_2) = x_1^2x_2 + t_12x_1x_2 + t_2x_1^2. \]
\end{example}

The following characterizes the images of the partial derivatives
in a coordinate free way.

\begin{lemma} \label{lem:DefoDerivation}
Let $I \sus \kr[X]$ be an ideal. Then $\tilde{I} = \tau(I)$
is a flat deformation of $I$ over $\kr[X]/(X^2)$. 
The corresponding linear map
$X^* \rarr \Hom(I, \kr[X]/I)$ sends a dual basis element $x_i^*$ to the 
derivation $\partial/\partial {x_i}$.
\end{lemma}

\begin{proof}
By tensoring with $\kr[X]$ we get a flat (trivial)
deformation $I \te_{\kr} \kr[X]$ in $\kr[X] \te_{\kr} \kr[X]$
over $\kr[X]$. Denote the generators of the algebra
$\kr[X] \te_{\kr} \kr[X] = \kr[X \oplus X]$ by $x_i \oplus 0$ and 
$0 \oplus t_i$.
We now perform the coordinate change $x_i \mapsto x_i \oplus t_i$
and $t_i \mapsto t_i$. Note that this coordinate
change is a $\kr[X]$-module map for the inclusion $x_i \mapsto 0 \oplus t_i$.
The image of $I \te_{\kr} \kr[X]$ by this coordinate change is 
denoted $\tilde{I}$.
Since 
\[ (\kr[X]/I) \te_{\kr} \kr[X] \rarr \kr[X \oplus X]/\tilde{I}\]
is a $\kr[X]$-module isomorphism, the right side will be flat over
$\kr[X]$. The result now follows by the base change 
$\kr[X] \rarr \kr[X]/(X)^2$.
\end{proof}

The cokernel $\coker \delta^\ast$ of the map $\delta^\ast$ in 
\eqref{eq:DecoDelta} is called the 
{\it first cotangent module} of $\kr[X]/I$ and is denoted by $T^1(\kr[X]/I)$. 
Its elements are in
one-to-one correspondence with the non-trivial infinitesimal deformations of 
$\kr[X]/I$. 

\begin{definition} The ideal $I$ is a {\it rigid ideal} if 
$T^1(\kr[X]/I)$ vanishes or equivalently the map $\delta^*$ is surjective.
\end{definition}
This means that every deformation of $I$ comes from a coordinate change.
We show in Section \ref{sec:Rigid} that $J(2,P) \sus B(2,P)$ is a rigid ideal.

\subsection{The first cotangent cohomology for letterplace ideals}
\label{subsec:DecoLP}

We consider a finite poset $P$.
Let $S$ be the polynomial ring
$\kr[x_{[2]\times P}]$. Recall that we denote the variable
$x_{i,p}$ as $p_i$. The letterplace ideal $L = L(2,P)$ is generated by
all monomials $p_1 q_2$ where $p \leq q$. We shall compute the 
first cotangent cohomology group $T^1(S/L)$, or rather its minimal
generating set as an $S$-module. The general theory for doing this
for Stanley-Reisner rings is developed in \cite{AltCh-Cotang}.
The results for $T^1(\kr[x_{[2] \times P}/L(2,P))$ are however so simple that we do
it from first principles, without recalling technicalities in loc.cit.

If $J$ and $U$ are subsets of $P$ we say that
$U$ is an upper bound set for $J$ 
if for every $p \in J$ there is a $q \in U$
with $p \leq q$. Similarly if $F$ and $D$ are subset of $P$, we say
$D$ is a lower bound set for $F$ if for every $p \in F$ there is a $r \in D$
with $r \leq p$.
For $p \in P$ let $J(< p)$ (resp. $J(\leq p)$) be the order ideal of $P$
consisting of all $r \in P$ with $r < p$ (resp. $r \leq p$), and let $F(> p)$ 
(resp. $F(\geq p)$) be the order filter consisting of all $q \in P$ with 
$q > p$ (resp. $q \geq p$). 
  
The first cotangent cohomology $T^1(S/L)$ is the quotient of 
$\Hom_S(L, S/L)$ by the derivations $\partial/\partial x_{i,p}$ as
we range over the variables of $S$. 
The following describes the minimal generators of $T^1(S/L)$. It
is remarkable in that they come from maps in $\Hom_S(L, S/L)$ where only 
one monomial is mapped to something nonzero. Also all monomials
$p_1q_2 \in I$ where $p <q$ are always mapped to zero.

\begin{theorem} \label{thm:DefCotang}
Given $p \in P$. Let $U$ be a an inclusion
minimal subset of $P - F(\geq p)$ which is an upper bound set for 
$J(< p)$, and let 
$D$ be an inclusion minimal subset 
of $P - J(\leq p)$ which is a lower bound set for $F(> p)$. 
Suppose also that whenever $r \in D$ and $s \in U$, we do not have $r \leq s$.
Then the map
\[ \phi : L \rarr S/L \] sending
\[ p_1 p_2 \mapsto \underset{r \in D}{\Pi} r_1 \underset{s \in U}{\Pi} s_2 \]
and all other minimal monomial generators in $L$ to zero, is a nonzero map of $S$-modules.

Moreover as we vary over $p \in P$ and sets $D$ and $U$ with this
property, the maps $\phi$ form a minimal generating set of the
cotangent cohomology $T^1(S/L)$.
\end{theorem}

Note that $D$ and $U$ above will both be antichains. We denote the
map above as 
\[ \phi = (p_1 p_2 \mapsto \underset{r \in D}{\Pi} r_1 
\underset{s \in U}{\Pi} s_2) . \]
Before proving the above theorem we develop some lemmata.

%Let $L$ be a monomial ideal generated in degree $d$.
%Let $M$ be the the matrix of relations of $L$. Any map $\phi:L \to S/L$ gives an element $\varphi$ of $\Hom_\kk(L_d,(S/L))$ by definition. Conversely, any $\kk$-linear map $\varphi:L_d \to (S/L)$ which satisfies the relations of $L$ algebraically extends to a well-defined graded $S$-linear map $\phi:L\to S/L$. Therefore there is a one-to-one correspondence
%$$\Hom(L,S/L)) \cong \{\phi\in \Hom_\kk(L_d,(S/L))| \phi \text{ satisfies the relations in }M\}.$$

\begin{lemma}
\label{lemma-01}
Let $\phi \in \Hom_S(L,S/L)$.
Modulo $\im \delta^\ast$ the homomorphism $\phi$ is equal to a homomorphism $\phi'$ with the property that for any $p\leq q$ in $P$, 
$\phi'(p_1q_2)$ 
can be written as a linear combination of monomials $m$ not 
divisible by $p_1$ or $q_2$.
%If for $p\leq q$ in $P$, $\phi(p_1q_2) = p_1m$ (or $q_2x$) for some monomial $m$ of degree $i-1$ then $\phi$ decomposes as $\phi = m \frac{\partial}{\partial q_2} + \phi'$ (or $m\frac{\partial}{\partial p_1} + \phi''$) where $\phi'$ (or $\phi''$) is in $\Hom(L,S/L)_i$.
\end{lemma}

\begin{proof}
Suppose $\phi(p_1q_2)$ contains a nonzero term of the form 
$cp_1m$ for some monomial $m$ and constant $c$.
We show that modulo $\im \delta^\ast$ we can eliminate this term.
More precisely we show that for any generator of $L$ of form $r_1q_2$, for some $r\leq q$, $\phi(r_1q_2)$ contains the term $cr_1m$. 
Therefore by subtracting $cm \frac{\partial}{\partial q_2}$ from $\phi$ we can eliminate the terms $cr_1m$ without adding any extra terms to $\phi(r_1q_2)$.

Suppose $r_1 q_2$ is a generator of $L$ that contains $q_2$. The syzygy $r_1(p_1q_2) - p_1(r_1q_2)$ induces the relation $r_1 \phi(p_1q_2) - p_1\phi(r_1q_2)$.
The term $cr_1p_1m$ either belongs to $L$ or for some 
$m'$ in $\phi(r_1q_2)$ it cancels out by the term $cp_1 m'$.
If $r_1p_1m$ is in $L$ then $r_1m \in L$. If it cancels out then $m' = cr_1 m$. 
Therefore in either case $\phi(r_1q_2)$ contains $cm \frac{\partial}{\partial q_2}(r_1q_2)$.
Note that if for some generator $g$ of $L$, $q_2 \nmid g$ then $m \frac{\partial}{\partial q_2}(g) = 0$.
The proof for the case where $\phi(p_1q_2)$ contains a term of the form $mq_2$ 
is similar.
\end{proof}

For any homomorphism $\phi \in \Hom_S(L,S/L)$, the following lemma shows which monomials can appear in the image of generators of $L$.

\begin{lemma}
\label{lemma-02}
Let $\phi \in \Hom_S(L,S/L)$.
For $p \leq q$ let $p_1 q_2$ be a generator of $L$.
We may write $\phi(p_1 q_2)$ as a linear combination of monomials $m$ 
that are not in $L$.
\begin{enumerate}
\item If $p \neq q$ then $\gcd(p_1q_2, m) \neq 1$.
\item If $p = q$ then either $\gcd(p_1p_2, m) \neq 1$ or, 
for all $r < p$ there is $t \geq r$ such that $t_2 | m$ and 
for all $r' > p$, there is $s \leq r'$ such that $s_1 | m$.
\end{enumerate}
\end{lemma}

\begin{proof}
\noindent 1.  Suppose $p \neq q$ and $\gcd(p_1q_2,m) = 1$. 
We have a relation $p_2 \phi(p_1q_2) - q_2\phi(p_1p_2)$ in $S/L$ since 
$p_2(p_1q_2) - q_2(p_1p_2)$ is a syzygy of $L$. 
If $p_2 m$ is canceled with some monomial in $q_2\phi(p_1p_2)$ then $q_2| m$,
against our initial assumption above. 

Hence $p_2m\in L$, and this implies that there is some $s \leq p$ such that 
$s_1 | m$. 
In a similar manner the syzygy $q_1(p_1q_2) - p_1(q_1q_2)$ shows that there is 
some $t \geq q$ such that $t_2 | m$. So $s_1 t_2 | m$ for $s \leq t$ 
which is a contradiction. The upshot is that we must have $\gcd(p_1q_2,m)$
must be nontrivial.

\medskip
\noindent 2. Suppose $p = q$ and $\gcd(p_1p_2,m)=1$. Let $x(p_1p_2)-y(p'_1q'_2)$ be a  relation involving the generator $p_1p_2$, so
$y$ contains at least one of $p_1$ or $p_2$ (and $x$ does not).
If $xm$ is going to be canceled by some monomial in $y\phi(p'_1q'_2)$ then $p_1p_2$ and $m$ can not be relatively prime. Therefore if $\gcd(p_1p_2,m)=1$ then for any such syzygy $xm$ belongs to $L$. Note that if $x(p_1q_2)-y(p'_1q'_2)$ is not a linear syzygy then $xm$ belongs to $L$, so we only need to consider linear syzygies.
For any $r < p$ the linear syzygy $r_1(p_1p_2) - p_1(r_1p_2)$ implies $r_1m \in L$. Hence for some $t \geq r$, $t_2$ divides $m$. Similarly for any $r' > p$, $m$ should be divided by a variable $s_1$ for some $s\leq r'$.
\end{proof}

\begin{proof}[Proof of Theorem \ref{thm:DefCotang}]
Let $\phi\in \Hom(L,S/L)$ be a homomorphism. By Lemmata \ref{lemma-01} and 
\ref{lemma-02} we can decompose $\phi$ as $\phi=\phi_1+\phi_2$ where $\phi_1$ belongs to the image of $\delta^\ast$ and for any $p\in P$, all the monomials in $\phi_2(p_1p_2)$ are relatively prime to $p_1p_2$ and also for any $p < q$, $\phi_2(p_1q_2)=0$.
Now let $m$ be a nonzero monomial in $\phi(p_1p_2)$, let $\psi$ be a map that sends $p_1p_2$ to $m$ and any other generator of $L$ to zero. Since $m$ is relatively prime to $p_1p_2$ the second part of proof of \ref{lemma-02} shows that this map satisfies all the relations of $L$. Hence it is a well-defined homomorphism.
By Lemma \ref{lemma-02} (2), there exists some $U$ and $D$ as in \ref{thm:DefCotang} such that $\Pi_{r\in D} r_1 \Pi_{s\in U} s_2|m$.
Note that if $U$ contains an element $s$ in $F(\geq p)$ then $D$ contains an element $r$ such that $r\leq s$ and $m$ belongs to $L$ which is a contradiction. Hence $U\subseteq P - F(\geq p)$. Similarly $D\subseteq P - J(\leq p)$.
Therefore the homomorphisms in \ref{thm:DefCotang} generate $T^1(S/L)$.
\end{proof}

\begin{example}
Let $P$ be a poset with following Hasse diagram.
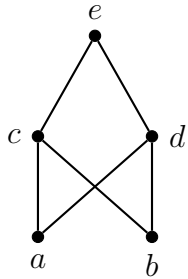
\begin{center}
 \begin{tikzpicture}[scale=1, vertices/.style={draw, fill=black, circle, inner sep=1.5pt}]
                \node [vertices, label=below:{$a$}] (0) at (-.75+0,0){};
                \node [vertices, label=below:{$b$}] (3) at (-.75+1.5,0){};
                \node [vertices, label=left:{$c$}] (1) at (-.75+0,1.33333){};
                \node [vertices, label=right:{$d$}] (2) at (-.75+1.5,1.33333){};
                \node [vertices, label=above:{$e$}] (4) at (-0+0,2.66667){};
        \foreach \to/\from in {0/1, 0/2, 1/4, 2/4, 3/1, 3/2}
        \draw [-,thick] (\to)--(\from);
  \end{tikzpicture}
\end{center}
For the point $a$, $J(<a) = \emptyset$ so the empty set is the only inclusion minimal subset of $P-F(\geq a)$ which is an upper bound for $J(<a)$. We also have $F(>a) = \{c,d,e\}$ and the inclusion minimal subsets of $P-J(\leq a)$ that are lower bounds for $F(>a)$ are $\{c,d\}$ and $\{b\}$.
Therefore for $a$, we have two first order deformations corresponding to maps
\[
a_1a_2 \mapsto c_1d_1, \qquad a_1a_2 \mapsto b_1.
\]
Consider the point $c$. $J(<c) = \{a,b\}$ and the minimal upper bounds in $P-F(\geq c)$ are $\{a,b\}$ and $\{
d\}$. We also have $F(>c) = \{e\}$ and the minimal lower bounds in $P-J(\leq c)$ are $\{d\}$ and $\{e\}$. Note that since $d_1d_2$ is in $L(2,P)$ we only have 3 maps corresponding to the point $c$ of $P$.
\[
c_1c_2 \mapsto a_2b_2d_1,\qquad c_1c_2 \mapsto a_2b_2e_1,\qquad c_1c_2 \mapsto d_2e_1.
\]
Finally, one can show that the first cotangent cohomology module is minimally generated by the following 11 maps.
\begin{align*}
& a_1a_2 \mapsto c_1d_1, \qquad a_1a_2 \mapsto b_1,\\
& b_1b_2 \mapsto c_1d_1, \qquad b_1b_2 \mapsto a_1,\\
& c_1c_2 \mapsto a_2b_2d_1, \quad c_1c_2 \mapsto a_2b_2e_1, \qquad c_1c_2 \mapsto d_2e_1,\\
& d_1d_2 \mapsto a_2b_2c_1, \quad c_1c_2 \mapsto a_2b_2e_1, \qquad c_1c_2 \mapsto c_2e_1,\\
& e_1e_2 \mapsto c_2d_2.
\end{align*}
\end{example}

\begin{corollary} \label{cor:CotangRP}
Suppose the Hasse diagram of $P$ is a rooted tree. Let $p$ in $P$
and $b^1, \ldots, b^m$ its children.
\begin{itemize}
\item Let $q$ be such that the meet of $q$ and $p$ is the parent of $p$.
The map sending $p_1p_1 \mapsto q_2 \prod_{i=1}^m b^i_1$ and
all other monomials to zero, is in $T^1(S/L)$.
\item Let $\rho$ be the root of $P$. The map sending $\rho_1\rho_2 \mapsto
\prod_{i = 1}^m b^i$, and all other monomials to zero, is in $T^1(S/L)$.
\end{itemize}
As $q$ and $p$ vary, these maps generate $T^1(S/L)$.
\end{corollary}

\section{Flatness of deformation family}
\label{sec:Flat}

We show that the ring $B(2,P)/J(2,P)$ is a flat deformation
of $\kr(2,P)/L(2,P)$ over the base ring $B = \kr[u_{\empt,\rot}, u_{q,p}]$. We do this in 
Theorem \ref{thm:Flat}, but before that we develop some auxiliary 
results.

\begin{proposition} \label{pro:FlatBasic} Given element $p, b,c \in P$
with $p \leq b$ and $p \leq c$. Then 
\[ S_p(b) c_2 - b_2 S_p(c) \]
is in $J(2,P)$. 
\end{proposition}

We shall prove this by induction on $\depth(p)$. For the below
lemmata we assume the above proposition holds for a given $p$, and
the lemmata are consquences of this.

\begin{lemma} \label{lem:FlatTS} Assume Proposition \ref{pro:FlatBasic} holds
for a given $p$. 
Let $q$ be sibling of $p$ (possibly equal) and $b \geq p$. 
Then 
\begin{equation}
S_pT_p(q) b_2 - T_p(q) S_p(b) \label{eq:FlatTS}
\end{equation} is in $J(2,P)$. 
\end{lemma}

\begin{proof} 
If $p = q$, then \eqref{eq:FlatTS} is 
\[ p_1b_2 - T(p) S_p(b) \] and so is in $J(2,P)$ by definition.
So assume $p \neq q$. Let $T_p(q) = - \sum_{c \geq p} c_2u_{c,q}$.
Since 
\[ S_p(c) b_2 - c_2 S_p(b) \] is in $J(2,P)$ we immediately get the lemma.
\end{proof}

\begin{lemma} \label{lem:FlatSTT}
Assume Proposition \ref{pro:FlatBasic} holds for a given $p$. 
Let $q,r$ and $p$ be siblings (some possibly equal).
Then 
\[ S_pT_p(q) T_p(r) - T_p(q) S_pT_p(r) \]
is in $J(2,P)$. 
\end{lemma}

\begin{proof} If all three are equal this clearly holds. 
Assume then that $p$ is distinct from either $q$ or $r$, say distinct
from $r$. Then $T_p(r) = - \sum_{b \geq p} b_2u_{b,r}$. The statement
then follows by Lemma \eqref{lem:FlatTS} above. 
\end{proof}

Let $b^1,\ldots,b^m$ be the children of $a = b^0$.
Let $b^i, b^j$ be elements of $\{ b^0, \ldots, b^m \}$ and 
$b^k$ element of $\{b^1, \ldots, b^m\}$. 
Let $M(a)^{b^i b^j}_{b^k}$ be the submatrix of $M(a)$ obtained
by deleting columns $i$ and $j$ and row $k$,
and define the signed determinant
\[ D(a)^{b^i b^j}_{b^k} = 
\begin{cases} (-1)^{i+j+k} |M(a)^{b^i b^j}_{b^k}| &  i < j \\

 (-1)^{i+j+k-1} |M(a)^{b^i b^j}_{b^k}| & i > j
\end{cases}
\]
More generally for a sequence of indices $\bi : i_0, \ldots, i_\ell$
from $0,1,\ldots,m$ let $|\bi|$ be its length $\ell+1$ and
$\sigma(\bi)$ the sign of the permutation  that puts them
in strictly increasing order. Similiarly for a sequence 
$\bk : k_1, \ldots, k_\ell$ from
$1,2,\ldots, m$, and let 
\[D(a)^{b^{\bi}}_{b^{\bk}} = 
(-1)^{|\bi| + |\bk| + \sigma(\bi) + \sigma(\bk)}
|M(a)^{b^{\bi}}_{b^{\bk}}|, \]
where the last matrix is obtained by deleting the columns from $\bi$
and the rows from $\bk$.
It may be verified that we may expand $D(a)^{b^{\bi}}_{b^{\bk}}$ along
a new row $r$ as
\[  D(a)^{b^{\bi}}_{b^{\bk}} = \sum_{c} S_c T_c(r) D(a)^{b^{\bi,c}}_{b^{\bk,r}}
\]
and similarly when expanding along a column.

Suppose now Lemma \ref{pro:FlatBasic} is proven for all $p$ with
$\depth(p) \leq n$. 
\begin{lemma} \label{lem:FlatDT} Let $\depth(a) \leq n+1$, and 
let $b,c,d$ be distinct children of $a$. 
\begin{itemize}
\item[1.] $\sum_{x  = b^1}^{b^m} D(a)^{bc}_x T_d(x)$ is in $J(2,P)$.
\item[2.] $\sum_{x  = b^1}^{b^m} D(a)^{bc}_x T_a(x)$ is in $J(2,P)$.
\item[3.] $\sum_{x  = b^1}^{b^m} D(a)^{ab}_x T_c(x)$ is in $J(2,P)$.
\end{itemize}
\end{lemma}

\begin{proof}
1. We expand the determinant $D_{x}^{bc}(a)$ by column $d$. 
We then get
\[ D_{x}^{bc}(a) T_d(x) = 
\sum_{y = b^1, y \neq x}^{b^m} D_{xy}^{bcd}(a) S_dT_d(y) T_d(x). \]
We now sum these over $x$. Thus for each pair $(x,y)$ both $xy$ and
$yx$ occur as indices of the determinant. But these determinants
will then be negatives of each other.
We obtain
\[ \sum_{x  = b^1}^{b^m} D(a)^{bc}_{x} T_d(x) 
= \sum_{x<y} D(a)^{bcd}_{xy}[S_dT_d(y)T_d(x) - T_d(y)S_dT_d(x)]. \]
By Lemma \ref{lem:FlatSTT} the last bracket is in the ideal $J(2,P)$. 

\medskip
2. We do as above but expand along the column $a$.
The sum in the statement is:
\[ \sum_{x  = b^1}^{b^m} D(a)^{bc}_{x} T_a(x) = 
\sum_{x < y} D(a)^{abc}_{xy}[S_aT_a(y)T_a(x) - T_a(y)S_aT_a(x)]. \]
But the expression in the last bracket is
$u_{a,y} a_2 u_{a,x} - u_{a,y}a_2 u_{a,x}$ which is zero.

\medskip
3. We now expand by column $c$.
Again in the same way we get that the sum in the statement is:
 \[ \sum_{x  = b^1}^{b^m} D(a)^{ab}_{x} T_a(x) = 
\sum_{x < y} D(a)^{abc}_{xy}[S_cT_c(y)T_c(x) - T_c(y)S_cT_c(x)]. \]
By Lemma \ref{lem:FlatSTT}, again the last bracket is in the ideal $J(2,P)$. 
\end{proof}

We are now in a position to prove Proposition \ref{pro:FlatBasic}.

\begin{proof}[Proof of Proposition \ref{pro:FlatBasic}]
We shall split into two cases. The first is when $p < c$ and $p < b$ 
strictly, and the second is when $p = b$ and $p < c$.

So consider the first case. Let $q$ and $r$ be chidren of $p$ such
that
\[ p < q \leq b , \quad p < r \leq c,\]
so the depths of $q$ and $r$ are less than the depth of $p$. 
Then 
\[ S_p(b) c_2 - b_2 S_p(c) = 
D(p)^q S_q(b) c_2 - b_2 D(p)^r S_r(c). \]

Now we expand $D(p)^q$ by its column $r$, and we expand $D(p)^r$ by
its column $q$. Then the above is:
\begin{equation} 
 S_q(b) \sum_{x = b^1}^{b^m} D(p)^{qr}_{x} S_rT_r(x) c_2 
- S_r(c) \sum_{x = b^1}^{b^m} D(p)^{rq}_{x} S_qT_q(x) b_2 
\end{equation}
Now we do a little trick by subtracting and adding the same terms,
to make this:
\begin{align} \label{eq:FlatDTS} 
 & S_q(b) \sum_{x = b^1}^{b^m} D(p)^{qr}_{x} (S_rT_r(x) c_2 - T_r(x)S_r(c))
+ S_r(c) \sum_{x = b^1}^{b^m} D(p)^{qr}_{x} (S_qT_q(x) b_2 -T_q(x) S_q(b)) \\ \notag
+ & S_q(b)S_r(c) \sum_{x= b^1}^{b^m} D(p)^{qr}_{x} (T_r(x) + T_q(x)).
\end{align}

By Lemma \ref{lem:FlatTS} the first two summands are in $J(2,P)$. 
For the lower sum note that
\[ T_r(x) + T_q(x) = - T_p(x) - \sum_{b^i \neq q,r} T_{b^i}(x). \]
By Lemma \ref{lem:FlatDT} this terms is also in the ideal $J(2,P)$.

\medskip

Now assume $p = b$. The change to the above is that we first
get  
\[ S_p(b) c_2 - b_2 S_p(c) = 
D(p)^q\cdot 1 \cdot c_2 - b_2 D(p)^r S_r(c), \]
so $S_q(b)$ is replaced by $1$. The sum \eqref{eq:FlatDTS} is replaced
by 
\begin{align*} \label{eq:FlatDTS2} 
 &  \sum_{x = b^1}^{b^m} D(p)^{pr}_{x} (S_rT_r(x) c_2 - T_r(x)S_r(c))
+ S_r(c) \sum_{x = b^1}^{b^m} D(p)^{qr}_{x} (S_pT_p(x) p_2 -T_p(x) \cdot 1) \\
+ & S_q(b)S_r(c) \sum_{x= b^1}^{b^m} D(p)^{pr}_{x}(T_r(x) + T_p(x)).
\end{align*}

The first and third summands are in $J(2,P)$ just as above. The 
paranthesis in the second term is:
\[ u_{p,x}p_2 - p_2u_{p,x} \]
and so the second term vanishes. Thus this expression is also in $J(2,P)$.
\end{proof}

\begin{proposition}
\label{pro:FlatP2}
Let $a \leq b$. 
Then 
\begin{equation} \label{eq:FlatC0}
a_1 T(b) - T(a) R(a,b) b_1 
\end{equation}
is in the ideal $J(2,P)$.
\end{proposition}

\begin{proof}
Suppose $p$ is the parent of $b$ and let $b = b^1,b^2, \ldots,b^m$ 
be the childern of $p$. Then
\begin{align*}
a_1 T(b) - T(a)R(a,b)b_1 & = a_1p_2u_{p,b} - \sum_{x=b^2}^{b^m} a_1 T_x(b) - b_1 T(a) R(a,p) D(p)^b
\end{align*}
Expanding $D(p)^b$ by row $b$, it is
\[ D(p)^b =  -u_{p,b} D(b)^{bp}_{b} + \sum_{x = b^2}^{b^m} 
S_xT_x(b) D(p)^{bx}_{b}. \]
We insert this in the above expression and it becomes:
\begin{align*}
a_1p_2 u_{p,b} - \sum_{x=b^2}^{b^m} a_1 T_x(b) - b_1 T(a) R(a,p) 
[u_{p,b} D(p)^{pb}_b
+ \sum_{x=b^2}^{b^m} S_xT_x(b) D(p)^{bx}_b]
\end{align*}
Since 
\[D(p)^p = b_1 D(p)^{pb}_b + \sum_{x=b^2}^{b^m} S_xT_x(b) D(p)^{px}_b, \]
the above equals
\begin{align*}
& a_1p_2u_{p,b} - \sum_{x=b^2}^{b^m} a_1 T_x(b) - T(a) R(a,p) u_{p,b} D(p)^p + 
\sum_{x=b^2}^{b^m} T(a) R(a,p) u_{p,b} S_xT_x(b) D(p)^{px}_b\\
& - \sum_{x=b^1}^{b^m} b_1 T(a) R(a,p) S_xT_x(b) D(p)^{bx}_b.
\end{align*}
Note that 
\[ u_{p,b}[a_1 p_2 - T(a) R(a,p)D(p)^p] \]
is in the ideal $J(2,P)$. Adding this to the above, the above will
modulo $J(2,P)$ be:
\begin{equation} \label{eq:FlatLig}
 - \sum_{x=b^2}^{b^m} a_1 T_x(b) + \sum_{x=b^2}^{b^m} T(a) R(a,p) S_xT_x(b) 
[u_{p,b} D(p)^{px}_b] - \sum_{x=b^2}^{b^m} b_1 T(a) R(a,p) S_xT_x(b) D(p)^{bx}_b
\end{equation}
Expanding $D(p)^x$ by row $b$ it is:
\[ D(p)^x = -u_{p,b}D(p)^{xp}_b + \sum_{\substack{y = b^2\\y \neq x}}^{b^m}
S_yT_y(b)D(p)^{xy}_{b} + b_1D(p)^{xb}_b.\]
%We again do a little trick of adding and subtracting this to the above
%and get:
Using this we replace $u_{p,b}D^{px}_b$ in \eqref{eq:FlatLig} and get
\begin{align} \label{eq:FlatLast}
& - \sum_{x=b^2}^{b^m} a_1 T_x(b) + \sum_{x=b^2}^{b^m} T(a) R(a,p) S_xT_x(b) 
[ D(p)^x - b_1 D(p)^{xb}_b - \sum_{\substack{y=b^2\\y\neq x}}^{b^m} 
S_yT_y(b) D(p)^{xy}_b ]\\
& - \sum_{x=b^1}^{b^m} b_1 T(a) R(a,p) S_xT_x(b) D(p)^{bx}_b \notag
\end{align}
Here 
\[ a_1 T_x(b) - T(a)R(a,p)S_xT_x(b)D(p)^x = a_1T_x(b) - T(a)S_a(T_x(b)) \]
is in the ideal $J(2,P)$.
Furthermore the terms with $xy$ superscripts cancels against
the terms with $yx$ superscripts. Thus \eqref{eq:FlatLast} is in the
ideal $J(2,P)$.
%& = \sum_{x=b^2}^{b^m} \sum_{\substack{y=b^2\\y\neq x}}^{b^m} T(a) R(a,p) S_xT_%x(b) S_yT_y(b) D(p)^{xy}_b
%\end{align*}
%Since $D(p)^{xy}_b$ and $D(p)^{yx}_b$ are negatives of each other, the above
%becomes zero.
\end{proof}

\begin{theorem} \label{thm:Flat} 
The ring $B(2,P)/J(2,P)$ is a flat deformation of
of $\kr[x_{[2]\times P}]/L(2,P)$ over the ring $B = \kr[u_{\empt,\rot}, u_{q,p}]$. 
\end{theorem}

\begin{proof}
By Proposition \ref{pro:GradedFlatlift}
it is enough to show that all relations between the generators of $L(2,P)$ 
lift to relations between the corresponding generators of $J(2,P)$. 

The relations of $L(2,P)$ are the following. 

\begin{itemize}
\item[1.] $ c_2 \cdot a_1 b_2 - b_2 \cdot a_1 c_2$ where $a \leq b$ and $a \leq c$.
\item[2.] $b_1 \cdot a_1c_2 - a_1 \cdot b_1 c_2$ where $a \leq b \leq c$. 
\end{itemize}

The monomial $a_1 b_2$ deforms to $a_1 b_2 - T(a) S_a(b_2)$, and similarly
for $a_1 c_2$. 
Let $m_B$ be the maximal ideal of $B = \kr[u_{\empt,\rot}, u_{q,p}]$
generated by the $u$'s.
For relations of type 1. it will be enough to show
that 
\begin{equation} \label{eq:FlatRelEn} 
c_2(a_1 b_2 - T(a) S_a(b_2)) - b_2(a_1c_2 - T(a) S_a(c_2)) 
\end{equation} 
is in $J(2,P) \cap (m_B \cdot B(2,P))$. But \eqref{eq:FlatRelEn}
is \[ T(a) [S_a(c) b_2 - c_2 S_a(b)]. \]
By Proposition \ref{pro:FlatBasic} the second factor is
in $J(2,P)$ and since $T(a)$ is in $m_B$ we are done.

Consider now relations of type 2. It is enough to show that 
\begin{equation} \label{eq:FlatRelTo}
b_1(a_1 c_2 - T(a) S_a(c_2)) - a_1 (b_1 c_2 - T(b) S_b(c_2))
\end{equation}
is in $J(2,P) \cap (m_B \cdot B(2,P))$. But \eqref{eq:FlatRelTo}
is 
\[ S_b(c_2) [ a_1 T(b) - T(a) R(a,b) b_1]. \]
By Proposition \ref{pro:FlatP2} below,  the second factor is in $J(2,P)$. It is also in
$m_B$ due to the $T$'s in the second factor, and so we are done.
\end{proof}

Recall the positive $\ZZ([2] \times P)$-grading on $B(2,P)$ from
Section \ref{sec:Grad}.
 
\begin{corollary} \label{cor:FlatHilbert}
Let $(L(2,P))$ be the ideal generated by $L(2,P)$ in $B(2,P)$.
The $\ZZ([2] \times P)$-graded rings $B(2,P)/J(2,P)$ and $B(2,P)/(L(2,P))$
have the same Hilbert function $h : \ZZ([2] \times P) \rarr \NN$.
\end{corollary}

\begin{proof}
We consider the situation in greater generality. 
We have an $A$-graded polynomial ring $S$,
and a homogeneous ideal $I \sus S$. Let $B = \kr[u_k]_{k \in K}$ be an 
$A$-graded polynomial ring and $J \sus S \te_{\kr} B$ a flat deformaion
of $I$ over $B$. Suppose the $A$-grading on $(S \te_{\kr} B) /J$ is positive.
Then we show that $(S \te_{\kr} B)/J$ and $(S \te_{\kr} B)/(I)$ have the same
Hilbert function.

Let $u_1$ be a variable in $B$, and $B^\prime = B/u_1\cdot B$.
Note that $u_1$ is $A$-homogeneous.
Consider $ 0 \rarr B \mto{u_1} B \rarr B^\prime \rarr 0$. Tensoring this exact sequence
with the rings
above and taking into account their flatness over $B$, we have exact sequences
\begin{align*}  0 \rarr (S \te_{\kr} B)/J \mto{u_1} & (S \te_{\kr} B)/J 
\rarr  (S \te_{\kr} B^\prime)/J^\prime  \rarr 0 \\
0 \rarr (S \te_{\kr} B)/(I) \mto{u_1} & (S \te_{\kr} B)/(I)
\rarr  (S \te_{\kr} B^\prime)/(I) \rarr 0
\end{align*}
By induction on the number of $u$-variables, 
we may assume that $(S \te_{\kr} B^\prime)/J^\prime$
and $(S \te_{\kr} B^\prime)/(I)$ have the same Hilbert function. Then
the same must be true for $(S \te_{\kr} B)/J$ and $(S \te_{\kr} B)/(I)$.
\end{proof}

\section{Rigidity of the deformation family}
\label{sec:Rigid}

We show that the ideal $J(2,P) \sus B(2,P)$ is a rigid ideal, meaning
that every deformation of $J(2,P)$ comes from a 
change of coordinates.

Let $Y$ be a finite-dimensional vector space and
$U \sus Y$ a subspace. We also denote $Y/U = X$. Fixing a splitting 
the reader may
think of $Y$ as $X \oplus U$. We shall use such a direct 
decomposition in the arguments but will not need it for our statements. 
The space $X$ may usually be thought of as the space
generated by the variables $p_1$ and $p_2$ as $p$ ranges over $P$,
and $U$ as the space of variables $u_{q,p}$ and $u_{\empt,\rho}$ 
from Section \ref{sec:Family}.
Our basic situation is an ideal $J \sus  \kr[Y]$ such that
$\kr[Y]/J$ is flat over $\kr[U]$. 
It is then a flat deformation of the ideal $I = J \te_{\kr[Y]} \kr[Y/U]$
in $\kr[Y/U]$. The situation we have in mind is when $J = J(2,P)$
and $I = L(2,P)$.

%\subsection{Infinitesimal deformations of $J$ and $I$ and conditions
%for rigidity}
%Let $J^\prime \sus \kr[Y] \te_\kr \kV$ be a flat deformation of
%$J$ over $\kV$. By the map $\kr[Y] \rarr \kr[Y/U]$ the image 
%of $J^\prime$ is
%an ideal $I^\prime \sus \kr[Y/U] \te \kV$.

%\begin{lemma} \label{lem:DefoSurj}
% The ideal $I^\prime$ is a flat deformation of $I$
%over $\kV$ corresponding to the composition
%\[ V \rarr \Hom_{\kr[Y]}(J, \kr[Y]/J) \rarr \Hom_{\kr[Y/U]}(I,\kr[Y/U]/I). \]
%\end{lemma}

%\begin{proof}
%Let the deformation $J^\prime$ over $\kV$ correspond to
%\[ f + \sum_i v_i^* f_i, \quad f \in J. \]
%The ideal $I^\prime$ then consists of
%\[\ov{f} + \sum_i v_i^* \ov{f_i}, \]
%where the bar denotes the image by the map $\kr[Y] \rarr \kr[Y/U]$.
%The map $f \mapsto f_i$ is a $\kr[Y]$-module homomorphism 
%$J \rarr \kr[Y]/J$. It then induces a $\kr[Y/U]$-module
%homomorphism $J / (U \cdot J) \rarr \kr[Y/U]/I$. 
%But since $J$ is a flat deformation of $I$,  we have 
%$J/(U \cdot J) = I$. Hence the ideal $I^\prime$ is 
%a flat deformation over $\kV$.
%\end{proof}

Now we take a base change $\kr[U] \rarr \kr[U]/(U)^2$ and get
%a flat deformation $J^\prime$ of $I$ over $\kr[U]/\amincomment{(U)^2}$.
an ideal 
\[ J \te_{\kr[U]} \kr[U]/(U)^2 \sus  \kr[Y]/(U)^2 \]
which is a flat deformation of $I \sus \kr[Y/U]$ over $\kr[U]/(U)^2$. Hence
by Proposition \ref{pro:DefoTangent} it induces a map 
\[ U^* \mto{\alpha} \Hom_{\kr[Y/U]}(I, \kr[Y/U]/I). \]
Recall the first cotangent cohomology group $T^1(\kr[Y/U]/I)$. 
It parametrizes the first order non-trivial deformations of $I$.
The following says that if the deformation of $I$ over $\kr[U]/(U)^2$ 
encompasses all non-trivial deformations, then all first order deformations of 
$I$ come from infinitesimal coordinate changes of $J \sus \kr[Y]$.

\begin{lemma} \label{lem:DefoYSurj}
Suppose the composition
\begin{equation*} \label{eq:DefoUHom}
U^* \mto{\alpha} 
\Hom_{\kr[Y/U]}(I,\kr[Y/U]/I)  \mto{\tau} T^1(\kr[Y/U]/I)
\end{equation*}
maps $U^*$ to a generating set for $T^1$. 

Then the image of the composition (the first map is the map from
Lemma \ref{lem:DefoDerivation})
\begin{equation} \label{eq:RigidBeta}
 Y^* \mto{\beta} \Hom_{\kr[Y]}(J, \kr[Y]/J) \mto{q} 
\Hom_{\kr[Y/U]}(I, \kr[Y/U]/I)
\end{equation}
is a generating set for $\Hom_{\kr[Y/U]}(I, \kr[Y/U]/I)$.
%Then $J \sus \kr[Y]$ is 
%a rigid ideal.
\end{lemma}

\begin{proof} Let $X = Y/U$ and fix a splitting $Y = X \oplus U$.
Let the $u_i$'s form a basis for $U$. Consider elements in the ideal $J$
written as 
\[ f = f_0 + \sum_i u_if_{i} +
\text{terms involving degree two monomial in the } u_i, \] 
where $f_0 \in I$ and the $f_i$ are in $\kr[X]$. 
Then the image of $u_i^* \in U^* \sus Y^*$ by $\beta$ is the map
sending 
\[ f \mapsto f_i + \text{ terms of degrees } \geq 1 \text{ in the } u_i. \]
The image of $u_i^*$ by $\alpha$ is the map sending
\[ f_0 \mapsto f_i. \]

%The image of $f$ in $J \te_{\kr[Y]} \kr[Y]/(U^2)$ is 
%\[ f_0 + \sum_i u_if_{i}. \]
%The image of $u_i^*$ by the first map in \eqref{eq:DefoUHom}
%is the map sending $f_0 \mapsto f_i$. 
Hence we see that restricted to $U^* \sus Y^*$ there is a commutative diagram
\[ \begin{CD} U^* &  @>{=}>> & U^* \\
@VV{\beta}V & &  @VV{\alpha}V \\
\Hom_{\kr[Y]}(J, \kr[Y]/J) & @>>> & \Hom_{\kr[Y/U]}(I, \kr[Y/U]/I)
\end{CD} \]
(We do not get a commutative diagram if we replace $U^*$ by
$Y^*$ at the upper left.)
Consider the map 
\[ \kr[X] \te_\kr (X^* \oplus U^*) = \kr[X] \te_\kr Y^* \mto{\hat{\beta}}
 \Hom_{\kr[X]}(I, \kr[X]/I) \]
coming from the composing the maps in \eqref{eq:RigidBeta}.
We have established that the map on the second factor identifies
with $\alpha$.
The cokernel of
\[ \kr[X] \te_{\kr} X^*  \rarr
%\rarr \Hom_{\kr[Y]}(J, \kr[Y]/J) \rarr 
\Hom_{\kr[X]}(I, \kr[X]/I). \]
%\comment{Shouldn't the first term above be $\kr[X] \otimes_\kr X^\ast$?}
is $T^1(\kr[X]/I)$. 
Hence we get the lower row in the diagram below, and a commutative diagram
\[ \begin{CD} \kr[X] \te_{\kr}X^* & @>>> & \kr[X] \te_{\kr} (X^* \oplus U^*) 
& @>>> & \kr[X] \te_{\kr} U^* \\
@VVV & & @VV^{\hat{\beta}}V & & @VV{\tau \circ \alpha}V \\ 
\im  (\kr[X] \te_{\kr}X^*) & @>>> & \Hom_{\kr[X]}(I, \kr[X]/I) & @>{\tau}>> & T^1 
\end{CD} \]

%\comment{Question 1: Is the diagram above coming from this?
%\[
%\xymatrix{
%\kr[X] \te_{\kr} X^* \ar[r] \ar[rd] & \Hom_{\kr[Y]}(J, \kr[Y]/J) \ar[r] & 
%\Hom_{\kr[X]}(I, \kr[X]/I) \ar[r] & T^1\\
%& Im(\kr[X] \te_{\kr}) X^* \ar[ru]  & &
%}\]
%}
Since the left and right maps are surjective, the middle one is
also by the snake lemma.
\end{proof}

\begin{theorem} \label{thm:Rigid}
Suppose $\kr[Y]$ has a positive grading by an abelian group and 
the subspace $U \sus Y$ and the ideal $J$ are homogeneous for this
grading.
If the composition
\[U^* \rarr  \Hom_{\kr[Y/U]}(I,\kr[Y/U]/I)  \rarr T^1(\kr[Y/U]/I)
\]
is surjective, then $J \sus \kr[Y]$ is a rigid ideal.
\end{theorem}

\begin{proof}
By the lemma above we know that the composition
\begin{equation*}
 Y^* \te \kr[Y] \mto{\hbeta} \Hom_{\kr[Y]}(J, \kr[Y]/J) \mto{q} 
\Hom_{\kr[Y/U]}(I, \kr[Y/U]/I)
\end{equation*} 
is surjective.
We shall show by induction on $U$ that whenever the composition
$q \circ \hbeta$ is surjective, then $\hbeta$ is surjective.
This will prove that $J$ is rigid.

Let $V$ be a subspace of $U$ of codimension one, homogeneous for the
grading. Let $\ov{J}$ be $J \te_{\kr[Y]} \kr[Y/V]$. 
By base extension $\kr[Y/V]/\ov{J}$ is flat over $\kr[Y/V]$.
We have a factorization
\begin{equation*}
 Y^* \te \kr[Y] \mto{\hbeta} \Hom_{\kr[Y]}(J, \kr[Y]/J) \mto{q_1} 
\Hom_{\kr[Y/V]}(\ov{J}, \kr[Y/V]/\ov{J}) \mto{r}
\Hom_{\kr[Y/U]}(I, \kr[Y/U]/I)
\end{equation*} 
By induction, it is sufficient to prove that $q_1 \circ \hbeta$ is
surjective. Let us rename $Y/V$ as $\oY$. 
Note that $U/V = (u)$ generated by one element. We must then
show that if the composition
\begin{equation*}
 Y^* \te \kr[Y] \mto{\hbeta_1} \Hom_{\kr[\oY]}(\oJ, \kr[\oY]/\oJ) \mto{r} 
\Hom_{\kr[\oY/(u)]}(I, \kr[\oY/(u)]/I)
\end{equation*} 
is surjective, then the map $\hbeta_1$ is surjective.
Let $\phi$ in $\Hom_{\kr[\oY]}(\oJ, \kr[\oY]/\oJ)$ map to zero in 
$\Hom_{\kr[\oY/(u)]} (I, \kr[\oY/(u)]/I)$. We first show that $\phi = u \psi$
for some $\psi \in  \Hom_{\kr[\oY]}(\oJ, \kr[\oY]/\oJ)$.
Let $\{ f_i \}$ be a generating set for $\oJ$ and suppose
$f_i \mapsto u g_i$ by $\phi$. Let $\sum_i r_if_i$ be a relation between the 
$f_i$. Then $ u \sum_i r_i g_i$ is zero in $\kr[\oY]/\oJ$. But due to flatness
of  $\kr[\oY]/\oJ$ over $\kr[u]$, the element $u$ is a nonzero divisor.
Hence $\sum_i r_i g_i$ is zero in $\oJ$ and so $f_i \mapsto g_i$
gives an element $\psi \in \Hom_{\kr[\oY]}(\oJ, \kr[\oY]/\oJ)$.
So there is an injection
\[ \Hom_{\kr[\oY]}(\oJ, \kr[\oY]/\oJ) \te_{\kr[\oY]} \kr[\oY/(u)] \hookrightarrow
\Hom_{\kr[\oY/(u)]} (I, \kr[\oY/(u)]/I). \]
But this must also be a surjection since the composition $r \circ \hbeta_1$
is surjective. Hence the above is an isomorphism, and so 
\[ \Hom_{\kr[\oY]}(\oJ, \kr[\oY]/\oJ) \te_{\kr[\oY]} \kr \iso
\Hom_{\kr[\oY/(u)]} (I, \kr[\oY/(u)]/I) \te_{\kr[\oY/(u)]} \kr . \]
We apply Nakayama's lemma applied to the finitely generated module
$\Hom_{\kr[\oY]}(\oJ, \kr[\oY]/\oJ)$ over the positively graded ring $\kr[\oY]$.
Since the image of $Y^*$ generates $\Hom_{\kr[\oY/(u)]} (I, \kr[\oY/(u)]/I)$
its image must then also generate $\Hom_{\kr[\oY]}(\oJ, \kr[\oY]/\oJ)$.
\end{proof}

\begin{corollary} \label{cor:Rigid} 
The ideal $J(2,P) \sus B(2,P)$ is a rigid ideal.
\end{corollary}

\section{The multigraded Hilbert schemes}

In the previous section we established that the ideals $J(2,P)$
were rigid. This is a property of (infinitesimal) deformation theory,
concerned with first order deformations.
This section is concerned with the global deformation family. 
The moduli spaces of families of quotient rings of a polynomial ring,
are the Hilbert schemes. We shall establish that the family of 
quotient rings $B(2,P)/J(2,P)$ by coordinate chages 
maps dominantly onto any component of 
the Hilbert scheme containing $L(2,P)$. In order to have a 
Hilbert scheme, we need a grading on the ring $B(2,P)$. 
We shall follow the most general such approach, that of Haiman and
Sturmfels considering multigraded Hilbert schemes \cite{HaSt}.
Continuing the setting of Section \ref{sec:Rigid},
we assume that $Y$ is graded by an abelian group $A$. 
Then $A$ gives a grading on the polynomial ring $\kr[Y]$. 
%Recall that this grading is said to be {\it positive}
%if the only elements in $\kr[Y]$ of degree $0$ are the constants in $\kr$.
We assume that $U$ is a homogeneous
subspace of $Y$, and that $J$ is a homogeneous ideal 
for this $A$-grading. The grading on $\kr[Y/U]/I$ is {\it admissible}
if for each degree $a$, the graded piece $(\kr[Y/U]/I)_a$ is 
a finite-dimensional vector space.

\subsection{The generic coordinate change}
Choose a finite subspace $T \sus \Hom_{\kr}(Y, \kr[Y/U])_0$.
%Note that the right side is a finite dimensional vector space,
%since the grading is positive.
We get a map 
\[ Y \rarr \kr[Y/U] \te_\kr T^* \rarr \kr[Y/U] \te_\kr \kr[T^*], \]
giving a morphism of algebras
\begin{equation} \label{eq:DefoEq10}
\kr[Y] \rarr \kr[Y/U] \te_\kr \kr[T^*].
\end{equation}
Each $t \in T$ corresponds to a point in $\Spec \kr[T^*]$ or a map 
$\kr[T^*] \rarr \kr$. This induces from \eqref{eq:DefoEq10} a map
$\kr[Y] \rarr \kr[Y/U]$. This is the natural map coming from 
$t \in \Hom_{\kr}(Y, \kr[Y/U])_0$.
The map \eqref{eq:DefoEq10} induces the map
\begin{equation*} \label{eq:HilbTau}
\tau : \kr[Y] \mto{\Delta} \kr[Y \oplus Y] \mto{\gamma} 
\kr[Y/U] \te_\kr \kr[Y]
\mto{\beta} \kr[Y/U] \te_\kr \kr[Y/U] \te_\kr \kr[T^*] \mto{\alpha} 
\kr[Y/U] \te_\kr \kr[T^*], 
\end{equation*}
where in the last map we have used the multiplication on $\kr[Y/U]$.
Let $\tJ = \tau(J)$. Note that the fiber over $0 \in \Spec \kr[T^*]$
is $I \sus \kr[Y/U]$. 
The ideal $\tJ \sus \kr[Y/U] \te_\kr \kr[T^*]$ is the ideal we
get from $J$ by performing a generic coordinate change of $J$
and then restricting to $\kr[Y/U]$.

\begin{example}
Let $Y$ be $\langle x_1,x_2,u \rangle$ where $x_1$ and $x_2$ have
degree $1$ and $u$ has degree $2$. The subspace $U$ is generated by $u$.
Let $T = \Hom(\langle x_1,x_2,u \rangle, \kr[x_1,x_2,u])_0$.
The space $T$ is eight-dimensional. A basis of $T$ are the maps
\begin{align*} t_{11} : \,\,  & x_1 \mapsto x_1, \quad x_2 \mapsto 0, 
\quad u \mapsto 0 \\
t_{12} : \,\, & x_1 \mapsto x_2, \quad x_2 \mapsto 0, \quad u \mapsto 0 \\
\cdots \\
t_{u,22} : \,\, & x_1 \mapsto 0, \quad x_1 \mapsto 0, \quad, u \mapsto x_2^2
\end{align*}
The basis of $T$ is (the meaning of the maps should be clear from the above)
\[ t_{11}, \, t_{12}, \, t_{21},\, t_{22},\, t_{u,u},\, t_{u,11},\, 
t_{u,12},\, t_{u,22}. \]
and their dual elements give a basis for $T^*$.
Denote the other copy of $Y$
by $\langle y_1,y_2,v \rangle$

We consider the homogeneous polynomial
$x_1^2x_2 + ux_1$ in $\kr[Y]$.
Applying the map $\tau$ this becomes 
\begin{align*} 
x_1^2x_2 + ux_1 \overset{\Delta} \mapsto & (x_1+ y_1)^2(x_2 + y_2) + 
(u+v)(x_1 + y_1) \\ 
\overset{\alpha \circ \beta \circ \gamma} \mapsto
& (x_1 + t_{11}^*x_1 + t_{12}^*x_2)^2(x_2 + t_{21}^*x_1 + t_{22}^*x_{22})\\
+ & (t_{u,11}^*x_1^2 + t_{u,12}^*x_1x_2 + t_{u,22}^*x_2^2)(x_1 + t_{11}^*x_1 + 
t_{12}^*x_2). 
\end{align*}
This is the form we get by performing a generic coordinate change
on the polynomial $x_1^2x_2 + ux_1$ and then taking the image in 
the quotient ring $\kr[Y/U] = \kr[x_1,x_2]$
\end{example}

\begin{proposition}  \label{pro:DefoOpenflat} 
%Suppose $Y$ has a positive grading by an abelian group and that
%$J$ is homogenous for this grading.
Suppose the $A$-grading on $\kr[Y/U]/I$ is admissible.
There is a localization $\kr[T^*]_f$ with $f(0) \neq 0$ such
that $\tJ_f \sus \kr[Y/U] \te_\kr \kr[T^*]_f$ is flat over $\kr[T^*]_f$.
\end{proposition}

\begin{proof}
%\comment{In the proof sometimes $J_2$ is denoted by $J'$, I wonder which notation you prefer.}
Let $X = Y/U$ and fix a splitting $Y/U \rarr Y$ such that 
$Y = X \oplus U$. 
%Denote by $A$ the abelian group giving
%the positive grading on $\kr[Y]$. 
The elements in $T$ have degree $0$. 
Note that for each degree $a \in A$ then $(\kr[X,T^*]/\tilde{J})_a$ is
a finitely generated $\kr[T^*]$-module.
 
%Let the $t_i's$ be a basis for $T$, and let the $y_i$'s 
%a basis for $Y$.
%and $u_i$'s
%constitute a basis for $X$ and $U$ respectively.
%Then if $t_i \mapsto \sum_j y_j^* \te_\kr m_{ij}$ 
%then \amincomment{$y_j \mapsto \sum_i t_i^* \te_\kr m_{ij}$}. 
%The $y_j$'s split into a basis for $X$, denoted by $x_j$'s,
%and a basis for $U$, denoted by $u_j$'s.

%The ideal $J^\prime$ is then obtained from the ideal $J$
%by the substitutions of variables:
%\[ x_j \mapsto x_j + \sum_i t_i m_{ij}, \quad
%u_j \mapsto \amincomment{\sum_i} t_i m_{ij}. \]

\medskip
Let $\mm$ be the maximal ideal in $\kr[T^*]$ corresponding to
the origin. We obtain a localized ideal $(\tJ)_{\mm} \sus \kr[X,T^*]_{\mm}$.
We first show that $((\kr[X,T^*]/\tJ)_a)_\mm$ is flat 
over the local ring $\kr[T^*]_{\mm}$.
By the local criterion of flatness, Theorem A.5 in \cite[App.A]{Ser}, 
it will be enough to show
that $\kr[X,T^*]/((T^*)^p + \tJ)_a$ is flat over
$\kr[T^*]/(T^*)^p$ for each natural number $p$.
(Strictly speaking this theorem
applies to the situation of a local homomorphism of local noetherian
rings, but the argument works just as well when the 
codomain ring has positive grading.)

Let the $f_k$ form a generating set for $J$.
We write $f_k^\prime$ for their images in $J/ (U)^p \cdot J \sus \kr[X,U]/(U)^p $.
Let $\ov{f_k}$ denote the image in $I \sus \kr[X]$, and
$\hat{f_k}$ the images of $f_k$ in $((T^*)^p + \tJ)/(T^*)^p$. 
Due to flatness of the deformation over $\kr[U]/(U)^p$, all relations
$0 = \sum_k \ov{r_k}\ov{f_k}$ in $\kr[X]$ lift to a relation 
$0 = \sum_k r_k^\prime f_k^\prime$ in $\kr[Y]/(U)^p$. Substituting 
\[ x_j \mapsto x_j + \sum_{\deg(m) = \deg(x_j)} t_{j,m}^* m, \quad
u_k \mapsto \sum_{\deg(u_k) = \deg(n)} t_{k,n}^* n, \]
we get a relation
\[ 0 = \sum_k \hat{r_k} \hat{f_k} \text{ in } \kr[X,T^*]/(T^*)^p \]
Hence all relations for $I$ lift to relations for $((T^*)^p + \tJ)/(T^*)^p$.
By Corollary A.11 in \cite[App.A]{Ser}
the quotient ring $\kr[X,T^*]/((T^*)^p + \tJ)$ is flat over $\kr[T^*]/(T^*)^p$.
Hence $\tJ_{\mm} \sus \kr[X,T^*]_{\mm}$ is a deformation of
$I$, flat over $\kr[T^*]_{\mm}$. 

\medskip
In the following we use: Let $M$ be a finitely generated module over an integral
domain $R$ and $p \in \Spec R$. i) If the localization $M_p$ is a free 
$R_p$ module of rank $r$, then there is an open subset $\Spec R_f \sus \Spec R$
containing $p$ 
such that $M_f$ is free of rank $r$ on $R_f$. ii) If for some $p \in \Spec R$, 
the fiber $M_{k(p)}$ is generated by $r^\prime$ elements, there is an open 
subset $\Spec R_g  \sus \Spec R$ containing $p$ such that $M_g$ is generated 
on $R_g$ by $r^\prime$ elements. Hence $r^\prime \geq r$ since 
$\Spec R_g$ and $\Spec R_f$ intersect nonempty. 
(They both contain the zero ideal).

We also use the following \cite[Prop.3.2]{HaSt}:
% see also
%condition (h') on p.742 in loc.cit.

\medskip
{\it Let $\kr[X]$ be an $A$-graded polynomial ring, 
and $h: A \rarr \NN$ a Hilbert function.
There is a finite subset of degrees $D \sus A$ such that
the following holds.
Let $I \sus \kr[X]$ be a monomial ideal  
which is i)  generated in degrees $D$ 
and ii) its Hilbert function $h_I$ has $h_I(a) = h(a)$ for 
$a \in D$.  Then $h_I(a) = h(a)$ for all $a \in A$.}

\medskip
Let $h : A \rarr \NN$ be the Hilbert function of $\kr[X]/I$, 
and let $D$ be a finite set of degrees given by the above
proposition. Every $((\kr[X,T^*]/\tJ)_a)_{\mm}$ is a free 
$\kr[T^*]_{\mm}$-module of finite rank $h(a)$. We may find an elements $f_a
\in \kr[T^*]$ with $f(0) \neq 0$ such that $((\kr[X,T^*]/\tJ)_a)_{f_a}$
is a free module of rank $h(a)$ for $a \in D$. Let $f = \prod_{a \in D} f_a$.

Suppose now there is some $a_0 \in A$ such that $((\kr[X,T^*]/\tJ)_{a_0})_{f}$
is not locally free. The free $\kr[T^*]_{\mm}$-module 
$((\kr[X,T^*]/\tJ)_{a_0})_{\mm}$ has rank $h(a_0)$. 
There must then be $t \in \Spec \kr[T^*]_f$ such
that the fiber $(\kr[X,T^*]/\tJ)_{\kr(t)}$ has dimension $> h(a_0)$ in 
degree $a_0$. We may write $(\kr[X,T^*]/\tJ)_{\kr(t)} = \kr(t)[X]/J^\prime$
where $J^\prime$ is the image of $\tJ_{\kr(t)}$ in $\kr(t)[X]$. 
%\comment{Question 2: Why such a $t$ exists? I did not understand this part.}
Fix a monomial order on $\kr(t)[X]$. Let $M$ be 
ideal generated by the initial monomials of the ideal of 
$J^\prime \sus \kr(t)[X]$ in degrees $D$. Then $h_M(a) = h(a)$
for $a \in D$, but 
\[ h_M(a_0) \geq \dim_\kr ((\kr(t)[X]/J^\prime)_{a_0} > h(a_0).\]
This contradicts \cite[Prop.3.2]{HaSt} given above. 
Hence $((\kr[X,T^*]/J)_{f})_a$
is a locally free $\kr[T^*]_f$ module of rank $h(a)$ for every degree
$a \in A$.
\end{proof}

To the flat family 
$\tJ \sus \kr[Y/U] \te_\kr \kr[T^*]_f$ we apply the base change
$\kr[T^*]_f \rarr \kr[T^*]/(T^*)^2$. 
%Let $J^\prime \sus \kr[Y/U] \te_\kr\kr[T^*]/
%(T^*)^2$ be the image of $J$ by this base change.

%Note that the map \eqref{eq:DefoEq10} 
%induces a map 
%\[ \kr[Y]/\amincomment{(Y)^2} \rarr \kr[Y/U] \te_\kr \kr[T^*]/\amincomment{(T^*%)^2}. \]
%Hence we get a map 
%\[ \kr[Y/U] \te_\kr \kr[Y]/\amincomment{(Y)^2} \rarr \kr[Y/U] \te_\kr \kr[T^*]/%\amincomment{(T^*)^2}. \]

\begin{corollary} \label{cor:DehiKTf}
 Let $J^\prime \sus \kr[Y/U] \te_\kr\kr[T^*]/
(T^*)^2$ be the image of $J$ by this base change.
The ideal $J^\prime$ is a flat deformation of $I$ over $\kr[T^*]/(T^*)^2$. 
\end{corollary}

Note that by Lemma \ref{lem:DecoTtoR} the map corresponding
to this deformation:
\[ T \rarr \Hom_{\kr[Y/U]}(I, \kr[Y/U]/I) \] 
is the map obtained from the composition
\[ Y^* \rarr \Hom_{\kr[Y]}(J, \kr[Y]/J) \rarr 
\Hom_{\kr[Y/U]}(I,\kr[Y/U]/I) \]
by using the $\kr[Y/U]$-module structure on the right module.

%\subsection{The multigraded 
%Hilbert schemes}

\subsection{The multigraded Hilbert scheme}

We recall the multigraded Hilbert scheme as introduced by
Haiman and Sturmfels in \cite{HaSt}.
As before $A$ is  an abelian group and $X$ a
finite dimensional $A$-graded vectorspace over $\kr$,
so the polynomial ring $\kr[X]$ becomes $A$-graded. 
The Hilbert schemes come about by introducing 
a Hilbert function $h: A \rarr \NN$.
There is then a Hilbert
scheme $H_{\kr[X]}^h$ parametrizing all ideals $I \sus \kr[X]$ such that
$\dim_\kr (\kr[X]/I)_a = h(a)$ for $a \in A$. 
More precisely, let $\kalg$ be the category of $\kr$-algebras. 
Then $H_{\kr[X]}^h$ is a $\kr$-scheme 
which represents the point functor
\[ \hat H_{\kr[X]}^h : \kalg \rarr \sets  \]
where $\hat H_{\kr[X]}^h(R)$ is the set of all ideals $J \sus R[X]$ such that
$(R[X]/J)_a$ is a locally free $R$-module of rank $h(a)$.
In particular an ideal $I \sus \kr[X ]$ with Hilbert function $h$
corresponds to a $\kr$-point $\Spec \kr \mto{p} H_{\kr[X]}^h$
of the Hilbert scheme.

We may restrict $\hat H_{\kr[X]}^h$ to the category
$\kartalg$ of local artinian $\kr$-algebras with residue field $\kr$.
If $A \rarr \kr$ is the augmentation map, it induces 
\begin{equation} \label{eq:HilbAk}
\hat H_{\kr[X]}^h(A) \rarr \hat H_{\kr[X]}^h(\kr).
\end{equation} 
Let $\hat H_{I \sus \kr[X]}^{h}(A)$ be the inverse image of the element 
$(I \sus \kr[X])$ in the right set of \eqref{eq:HilbAk}. 
Thus $\hat H_{I \sus \kr[X]}^{h}$ is a subfunctor of $\Def{I}$, giving the
{\it $A$-graded deformations} of $I$. Again we may restrict 
$\hat H_{I \sus \kr[X]}^h $ to
$\kvect$ and we get the following specialization
of Proposition \ref{pro:DefoTangent}.

\begin{proposition} \label{pro:HilbTangent}
There is a natural isomorphism of functors between
\[ \hat H_{I \sus \kr[X]}^h , \, \, \Hom_{\kr}(-, \Hom_{\kr[X]}(I,\kr[X]/I)_0) : 
\kvect \rarr \sets .\]
Since $\hat H_{\kr[X]}^h$ on $\kalg$ is represented by the Hilbert 
scheme $H_{\kr[X]}^h$,the tangent space of this scheme at $I$ is 
$\Hom_{\kr[X]}(I,\kr[X]/I)_0$. 
\end{proposition}

%\begin{proposition} 
%Let $V$ be a finite-dimensional vector space, and $V^*$ its dual.
%There is a one-to-one
%correspondence between:
%\begin{itemize}
%\item Linear maps $V^* \rarr \Hom_{\kr[X]}(I, \kr[X]/I)_0$
%\item Elements of $D(\kr[V]/(V^2))$, i.e. $A$-graded 
%flat deformations of $I$ over
%$\kr[V]/(V^2)$.
%\item Morphisms $\Spec \kV \rarr H_{\kr[X]}^h$. 
%\end{itemize}
%In particular the tangent space of $D$ and of the Hilbert scheme ident%ifies as 
%$\Hom_{\kr [X]}(I,\kr[X]/I)_0$.
%\end{proposition}

We now assume that $U \sus Y$ is graded by the abelian group $A$, 
and that
$I$ and $J$ are $A$-graded ideals in $\kr[Y/U]$ resp. $\kr[Y]$.
We assume the $A$-grading is admissible on 
$\kr[Y/U]/I$ and so we have a Hilbert function:
\[ h : A \rarr \NN, \quad h(a) = \dim_\kr (\kr[Y/U]/I)_a. \]
There is a map
\[ \Hom_{\kr}(Y, \kr[Y/U])_0 \rarr \Hom_{\kr}(Y, \kr[Y/U]/I)_0. \]
The right $\Hom$-space is a finite dimensional vector
space due to the grading being admissible.

\begin{theorem}
\label{thm:HilbSmooth} Suppose the $A$-grading on $\kr[Y/U]/I$ is admissible
and that the composition
\[ U^* \rarr 
\Hom_{\kr[Y/U]}(I,\kr[Y/U]/I)  \rarr T^1(\kr[Y/U]/I)\]
maps $U^*$ to a generating set of $T^1$.
Let $T$ be a finite dimensional subspace of $\Hom_{\kr}(Y, \kr[Y/U])_0$
which maps surjectively to $\Hom_{\kr}(Y, \kr[Y/U]/I)_0$.

Then the induced morphism from Theorem \ref{pro:DefoOpenflat}
\[ \Spec \kr[T^*]_f \rarr H_{\kr[Y/U]}^h \]
is surjective on tangent spaces at the origin $0 \in \Spec \kr[T^*]_f $.
\end{theorem}

\begin{proof}
Let $\mm \sus \kr[T^*]_f$ be the maximal ideal corresponding to the
origin in $\kr[T^*]_f$. The morphism
\[ \Spec \kr[T^*]_f/\mm^2 \rarr H_{\kr[Y/U]}^h \] is then
\[ \Spec \kr[T^*]/(T^\ast)^{2} \rarr H_{\kr[Y/U]}^h. \]
By Proposition \ref{pro:HilbTangent} this corresponds to the map of
tangent spaces
\begin{equation} \label{eq:DefoT0} T \rarr \Hom_{\kr[Y/U]}(I, \kr[Y/U]/I)_0.
\end{equation}
By the comment after Corollary \ref{cor:DehiKTf} and Lemma
\ref{lem:DecoTtoR} this map is 
obtained from the composition
\[ Y^* \rarr \Hom_{\kr[Y]}(J,\kr[Y]/J) \rarr \Hom_{\kr[Y/U]}(I,\kr[Y/U]/I) \]
as
\[ T \rarr \Hom(Y,\kr[Y/U]/I)_0 = (Y^* \te \kr[Y/U]/I)_0  \rarr
\Hom_{\kr[Y/U]}(I,\kr[Y/U]/I)_0. \]
Since the right map above is surjective by 
Lemma \ref{lem:DefoYSurj} the map \eqref{eq:DefoT0} on 
tangent spaces, is surjective.
\end{proof}

\begin{corollary}
The ideal $I$ is a smooth point on the Hilbert 
scheme $H_{\kr[Y/U]}^h$. The morphism $\Spec \kr[T^*]_f \rarr H_{\kr[Y/U]}^h$ 
is dominant on the component of the Hilbert scheme containing $I$.
So there is an open subset of the Hilbert scheme $H_{\kr[Y/U]}^h$
such that the ideals in this open subset are obtained from $J \sus \kr[Y]$
by a coordinate change, and then restricting to $\kr[Y/U]$.
\end{corollary}

Let $Y \sus \hY$ be an inclusion of finite-dimensional $A$-graded 
vector spaces. Consider the ideal $(J) \sus \kr[\hY]$ generated
by $J$. Note that this identifies as $J \te_{\kr[Y]} \kr[\hY] 
\sus \kr[\hY]$. Similarly we get an ideal $(I) \sus \kr[\hY /U]$.
Let  $\hat h$ denote the Hilbert function of the quotient ring, 
if $h$ is the Hilbert function of the quotient ring of $I \sus \kr[Y/U]$.
The cotangent cohomology 
\[ T^1(\kr[\hY/U]/(I)) = T^1(\kr[Y/U]/I) \te_{\kr[Y]} \kr[\hY]. \]
The theorem and corollary above still applies to this situation.

\medskip
\noindent {\bf Applications.}
Let $\ZZ([2] \times P) \rarr A$ be a homomorphism of abelian groups.
We take $Y$ to be the space generated 
%by the variables
%$x_{i,p}, (i,p) \in [2] \times P$ and $of 
by the linear forms of $B(2,P) = \kr[x_{[2]\times P}]
\te_{\kr} \kr[u_{\empt,\rho}, u_{q,p}]$, so $Y$ is generated by the $x$ and 
$u$-variables in this ring, and $U$ the linear forms in 
$\kr[u_{\empt,\rho}, u_{q,p}]$.
%Denote by $X_0$ be the space generated by the 
%$x$-variables. 

\medskip
\noindent 1. 
If $A$ gives an admissible grading on $\kr[X]/L(2,P)$, then
for an open subset of the Hilbert scheme $H^h_{\kr[X]}$, the ideals in this
open subset come from a 
change of coordinates in $J(2,P)$ and then restricting the ideal to 
$\kr[X]$. 

\medskip
\noindent 2.
Let $Y \sus \hY$ be such that $\hY/Y \iso U$ as $A$-graded spaces, or
equivalently $\hY/U \iso Y$. (For instance we take a copy $V$ of $U$
and let $\hY = Y \oplus V$.) Suppose $\kr[\hY/U]/(L(2,P))$ is admissible
for the $A$-grading 
(which by Corollary \ref{cor:FlatHilbert}
is equivalent to $\kr[Y]/J(2,P)$ being admissible). 
Then there is an open subset of 
the Hilbert scheme $H^{\hat h}_{\kr[\hY/U]}$ (which identifies as 
$H^{\hat h}_{\kr[Y]}$) such that the ideals in this open subset are obtained
from a coordinate change of $(J(2,P)) \sus \kr[\hY]$ and then
restricting to $\kr[\hY/U]$. But since $(J(2,P))$ is generated by 
$J(2,P) \sus \kr[Y]$ and $\hY/U \iso Y$, this identifies as a 
coordinate change of $J(2,P)$, and so the Hilbert scheme component
of $H^{\hat h}_{\kr[Y]}$ has an open subset consisting of ideals obtained
from $J(2,P)$ by coordinate changes. 

\medskip
\noindent 3. Often the $A$-grading on $\kr[\hY/U]/(L(2,P))$ is not admissible,
but there is a space $Y \sus Y^+ \sus \hY$ such that 
$\kr[Y^+/U]/(L(2,P))$ is admissible.
For instance take $A = \ZZ$
and let the $p_i$ map to positive values in $\ZZ$. Some of the $u$-variables
typically map to negative values. Then $\kr[Y]$ is infinite-dimensional
in degree $0$.  
Let $\hY = Y \oplus V$ where $V$ is a copy of $U$ and let 
$U^+ \sus U$ and $U^- \sus U$ be the variables of non-negative and negative
degrees, and similarly for $V$.
Let 
\[ J^\prime(2,P) = J(2,P) \te_{\kr[Y]} \kr[Y/U^-] \]
so we are setting the $u$-variables of negative degrees equal to zero.
Let $Y^+ = Y \oplus V^+$. Then there is an open subset of the Hilbert scheme of 
$(L(2,P)) \sus \kr[Y^+/U]$ where the ideals in this open subset are
obtained from $(J(2,P)) \sus \kr[Y^+]$ by coordinate changes and 
then restricting to $\kr[Y^+/U]$. Since $Y^+/U$ only has elements
of positive degree, all elements of negative degree in $\kr[Y^+]$ map to zero.
Also $(J(2,P))$ is generated by $J(2,P) \sus \kr[Y]$. 
Since $Y/U^- \iso Y^+/U$, this is then really
simply a coordinate change of $J^\prime(2,P)$. 

\begin{example}
Let $P$ be the star poset with unique minimal element $a$ and $3$ 
maximal elements $b,c,d$.
Consider the standard $\mathbb{Z}$-grading on $\kr[x_{[2]\times P}]$.
Since $\deg(u_{\emptyset,a}) = a_1 + a_2 - b_1 - c_1 - d_1$ 
with respect the standard grading on $\kr[x_{[2]\times P}]$ we have $\deg(u_{\emptyset,a}) = -1$.
For any other $u_{q,p}$, we have $\deg(u_{q,p})=1$. 
The space $U^-$ is generated by
$u_{\empt,a}$,
the grading is positive (and so admissible) on 
$\kr[Y/V] = \kr[x_{[2] \times P}, u_{a,b},u_{a,c},u_{b,c},u_{c,b}]$.
Let $J^\prime \sus \kr[Y/U^-]$ be the ideal obtained from $J$ by
setting $u_{\empt,a} = 0$. This is the ideal generated by the 
forms \eqref{eq:ExStarForms} after setting $u_{\empt,a} = 0$. 
Note that $Y^+ = Y \oplus V^+$, and $Y^+/U \iso Y/U^-$ as 
$A$-graded spaces. 
Then there is an open subset of the Hilbert scheme component of 
$(L(2,P)) \sus \kr[Y^+/U]$, such that the ideals in this open
subset come from a coordinate change in $J^\prime \sus \kr[Y/U^-]$.

%Let $T$ be the finite subspace of $\Hom_\kr(Y,\kr[X])$ which maps surjectively %to $\Hom_\kr(Y,\kr[X]/I)_0$. Since any homomorphism in $T$ maps $u_{\emptyset,a%}$ to zero, $T$ is in one-to-one correspondence with the $17\times 8=136$-dimen%sional space
%[
%\Hom_\kr(Y\backslash \{u_{\emptyset,a}\},X).
%\]
%Under the map $T^\ast \to \Hom_{\kr[X]}(L,\kr[X]/L)$, the 136-dimensional space% $T^\ast$ maps surjectively to the 80-dimensional tangent space $\Hom_{\kr[X]}(L,\kr[X]/L)$ at the point corresponding to $L$ on the Hilbert scheme.

\medskip
Now consider another $\mathbb{Z}$-grading on $\kr[x_{[2] \times P}]$. 
Let $a_1$ have degree $2$ and all the other variables have degree $1$. 
In this case $\deg(u_{\emptyset,a})=0$ and any homomorphism 
in $T$ maps $u_{\emptyset,a}$ to a scalar.
In this case there is open subset of the Hilbert scheme component of 
$(L(2,P)) \sus \kr[\hY/U]$ such that the ideals in this open subset come from
a coordinate change of $J(2,P) \sus \kr[Y]$, where $Y \iso \hY/U$ 
and $u_{\empt,a}$ will always be sent
to a scalar.
%Here $T$ is in one-to-one correspondence with the vector space
%\[
%\Hom_\kr(<u_{\emptyset,a}>,\kr) \oplus
%\Hom_\kr(Y\backslash \{u_{\emptyset,a},a_1 \},X\backslash \{a_1\})\oplus
%\Hom_\kr(<a_2>,(\kr[X]/L)_2)
%\]
%The dual of this $1+16\times 7 + 26 = 139$-dimensional space maps surjectively to the 95-dimensional tangent space $\Hom_{\kr[X]}(L,\kr[X]/L)$.
\end{example}

\medskip
\noindent {\bf Conclusion} 
The letterplace ideal $I = L(2,P)$ is usually not rigid, but we see that 
something nearly
as good holds when the Hasse diagram of the poset $P$ has tree structure. 
There is a ``lifting'' to a rigid ideal $J = J(2,P)$, and 
for an open set of the Hilbert scheme component of $L(2,P)$, all the
ideals come from a coordinate change of $J(2,P)$.

We have also done computations investigating simple cases when $P$
does not have tree structure, f.ex. the four element diamond poset.
It seems everything we show in this article also goes through.
We therefore make the following conjecture.

\begin{conjecture} \label{con:HilbCon}
For any finite poset $P$, the letterplace ideal $L(2,P) 
\sus \kr[x_{[2] \times P}]$ deforms
to a rigid ideal $J(2,P)$ in a ring $\kr[x_{[2] \times P}] \te_{\kr} \kr[U]$, 
with the quotient ring by the ideal
flat over a polynomial base ring $\kr[U]$. 
The ring $\kr[x_{[2] \times P}] \te_{\kr} \kr[U]$ is naturally 
positively graded by $\ZZ([2] \times P)$ with $J(2,P)$ homogeneous
for this grading.
\end{conjecture}

%such that for any
%$\kr$-algebra $R$ and $B$-graded ideal $J \sus R[X]$ such that
%$(R[X]/J)_a$ is a locally free free $R$-module of rank $h(a)$, there
%is a unique morphism $\Spec R \rarr 

\section{Appendix}

\begin{example}
Let $P$ be a poset with following Hasse diagram.
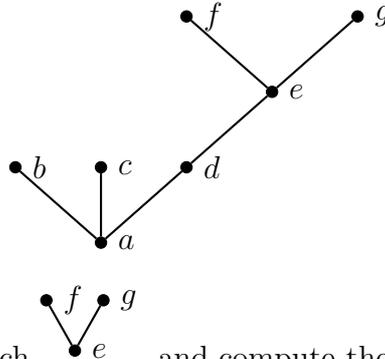
\begin{center}
\begin{tikzpicture}[scale=.75, vertices/.style={draw, fill=black, circle, inner sep=1.5pt}]
             \node [vertices, label=right:{$a$}] (0) at (-0+0,0){};
             \node [vertices, label=right:{$b$}] (1) at (-1.5+0,1.33333){};
             \node [vertices, label=right:{$c$}] (2) at (-1.5+1.5,1.33333){};
             \node [vertices, label=right:{$d$}] (3) at (-1.5+3,1.33333){};
             \node [vertices, label=right:{$e$}] (4) at (3,2.66667){};
             \node [vertices, label=right:{$f$}] (5) at (1.5,4){};
             \node [vertices, label=right:{$g$}] (6) at (3+1.5,4){};
     \foreach \to/\from in {0/1, 0/2, 0/3, 3/4, 4/5, 4/6}
     \draw [-,thick] (\to)--(\from);
\end{tikzpicture}
\end{center}
We start from the outer branch \begin{tikzpicture}[scale=.5, vertices/.style={draw, fill=black, circle, inner sep=1.5pt}]
             \node [vertices, label=right:{$e$}] (0) at (-0+0,0){};
             \node [vertices, label=right:{$f$}] (1) at (-.75+0,1.33333){};
             \node [vertices, label=right:{$g$}] (2) at (-.75+1.5,1.33333){};
     \foreach \to/\from in {0/1, 0/2}
     \draw [-,thick] (\to)--(\from);
     \end{tikzpicture} and compute the deformed relations.

Since $f$ is a maximal point we have
\[f_1f_2 - T(f) = f_1f_2 - u_{e,f} e_2 - u_{g,f} g_2\]
and similarly for $g$ we get the relation $g_1g_2 - u_{e,g} e_2 - u_{f,g} f_2$.
By definition
\[
M(e) =
\begin{bmatrix}
-u_{e,f} & f_1 & -u_{g,f}\\
-u_{e,g} & -u_{f,g} & g_1
\end{bmatrix}
\]
and
\[
e_1e_2 - T(e) D(e)^e = e_1e_2 -u_{d,e} d_2( f_1g_1-u_{g,f}u_{f,g}).
\]
Similarly we have $e_1f_2 - T(e)D(e)^f =  e_1f_2  - g_1d_2u_{e,f}u_{d,e} - d_2u_{g,f}u_{e,g}u_{d,e}$.
For $d_1f_2$ we have
\[
d_1f_2 - T(d)D(d)^e D(e)^f = d_1f_2 - 
(u_{a,d}a_2 + u_{b,d} b_2 + u_{c,d} c_2)u_{d,e} (g_1u_{e,f}+u_{ g,f}u_{e,g})
\]
For the deformed relation of $a_1f_2$ we need the matrix $M(a)$ for the branch point $a$.
The matrix $M(a)$ is given by
\[
\begin{bmatrix}
-u_{a,b} & b_1 & -u_{c,b} & -(u_{d,b}e_1 + u_{e,b}u_{d,e}D(e)^e+u_{f,b}u_{d,e}D(e)^e+u_{f,b}u_{d,e}D(e)^f+u_{g,b}u_{d,e}D(e)^g)\\
-u_{a,c} & -u_{b,c} & c_1 & -(u_{d,c}e_1 + u_{e,c}u_{d,e}D(e)^e+u_{f,c}u_{d,e}D(e)^e+u_{f,b}u_{d,e}D(e)^f+u_{g,c}u_{d,e}D(e)^g)\\
-u_{a,d} & -u_{b,d} & -u_{c,d} & d_1
\end{bmatrix}
\]
\end{example}
Now we have
\begin{align*}
a_1f_2 - T(a) S_a(f) & = a_1f_2 -u_{\emptyset,a} D(a)^d D(d)^e D(e)^f D(f)^f\\
& = a_1f_2 - u_{\emptyset,a} D(a)^d u_{d,e} (u_{e,g}u_{g,f}+u_{e,f}g_1)
\end{align*}

Analogously, we can compute the remaining deformed relations and get the following flat family.

\begin{footnotesize}
\begin{enumerate}[leftmargin=.5cm]
\item ${a}_{1} {a}_{2}-{b}_{1} {c}_{1} {d}_{1} u_{\emptyset,a} + {b}_{1} {e}_{1} u_{\emptyset,a} u_{d,c} u_{c,d} + {b}_{1} {f}_{1} {g}_{1} u_{\emptyset,a} u_{e,c} u_{c,d} u_{d,e}+{b}_{1} {f}_{1} u_{\emptyset,a} u_{g,c} u_{c,d} u_{d,e} u_{e,g}
+{b}_{1} {g}_{1} u_{\emptyset,a} u_{f,c} u_{c,d} u_{d,e} u_{e,f}-{b}_{1} u_{\emptyset,a} u_{e,c} u_{c,d} u_{d,e} u_{g,f} u_{f,g}+{b}_{1} u_{\emptyset,a} u_{f,c} u_{c,d} u_{d,e} u_{g,f} u_{e,g}+{b}_{1} u_{\emptyset,a} u_{g,c} u_{c,d} u_{d,e} u_{e,f} u_{f,g}+{c}_{1} {e}_{1} u_{\emptyset,a} u_{d,b} u_{b,d}+{c}_{1} {f}_{1} {g}_{1} u_{\emptyset,a} u_{e,b} u_{b,d} u_{d,e}+{c}_{1} {f}_{1} u_{\emptyset,a} u_{g,b} u_{b,d} u_{d,e} u_{e,g}+{c}_{1} {g}_{1} u_{\emptyset,a} u_{f,b} u_{b,d} u_{d,e} u_{e,f}-{c}_{1} u_{\emptyset,a} u_{e,b} u_{b,d} u_{d,e} u_{g,f} u_{f,g}+{c}_{1} u_{\emptyset,a} u_{f,b} u_{b,d} u_{d,e} u_{g,f} u_{e,g}+{c}_{1} u_{\emptyset,a} u_{g,b} u_{b,d} u_{d,e} u_{e,f} u_{f,g}+{d}_{1} u_{\emptyset,a} u_{c,b} u_{b,c}+{e}_{1} u_{\emptyset,a} u_{c,b} u_{d,c} u_{b,d}+{e}_{1} u_{\emptyset,a} u_{d,b} u_{b,c} u_{c,d}+{f}_{1} {g}_{1} u_{\emptyset,a} u_{c,b} u_{e,c} u_{b,d} u_{d,e}+{f}_{1} {g}_{1} u_{\emptyset,a} u_{e,b} u_{b,c} u_{c,d} u_{d,e}+{f}_{1} u_{\emptyset,a} u_{c,b} u_{g,c} u_{b,d} u_{d,e} u_{e,g}+{f}_{1} u_{\emptyset,a} u_{g,b} u_{b,c} u_{c,d} u_{d,e} u_{e,g}+{g}_{1} u_{\emptyset,a} u_{c,b} u_{f,c} u_{b,d} u_{d,e} u_{e,f}+{g}_{1} u_{\emptyset,a} u_{f,b} u_{b,c} u_{c,d} u_{d,e} u_{e,f}-u_{\emptyset,a} u_{c,b} u_{e,c} u_{b,d} u_{d,e} u_{g,f} u_{f,g}+u_{\emptyset,a} u_{c,b} u_{f,c} u_{b,d} u_{d,e} u_{g,f} u_{e,g}+u_{\emptyset,a} u_{c,b} u_{g,c} u_{b,d} u_{d,e} u_{e,f} u_{f,g}-u_{\emptyset,a} u_{e,b} u_{b,c} u_{c,d} u_{d,e} u_{g,f} u_{f,g}+u_{\emptyset,a} u_{f,b} u_{b,c} u_{c,d} u_{d,e} u_{g,f} u_{e,g}+u_{\emptyset,a} u_{g,b} u_{b,c} u_{c,d} u_{d,e} u_{e,f} u_{f,g}$,
\item ${a}_{1} {b}_{2}-{c}_{1} {d}_{1} u_{\emptyset,a} u_{a,b}-{c}_{1} {e}_{1} u_{\emptyset,a} u_{d,b} u_{a,d}-{c}_{1} {f}_{1} {g}_{1} u_{\emptyset,a} u_{e,b} u_{a,d} u_{d,e}-{c}_{1} {f}_{1} u_{\emptyset,a} u_{g,b} u_{a,d} u_{d,e} u_{e,g}-{c}_{1} {g}_{1} u_{\emptyset,a} u_{f,b} u_{a,d} u_{d,e} u_{e,f}+{c}_{1} u_{\emptyset,a} u_{e,b} u_{a,d} u_{d,e} u_{g,f} u_{f,g}-{c}_{1} u_{\emptyset,a} u_{f,b} u_{a,d} u_{d,e} u_{g,f} u_{e,g}-{c}_{1} u_{\emptyset,a} u_{g,b} u_{a,d} u_{d,e} u_{e,f} u_{f,g}-{d}_{1} u_{\emptyset,a} u_{c,b} u_{a,c}+{e}_{1} u_{\emptyset,a} u_{a,b} u_{d,c} u_{c,d}-{e}_{1} u_{\emptyset,a} u_{c,b} u_{d,c} u_{a,d}-{e}_{1} u_{\emptyset,a} u_{d,b} u_{a,c} u_{c,d}+{f}_{1} {g}_{1} u_{\emptyset,a} u_{a,b} u_{e,c} u_{c,d} u_{d,e}-{f}_{1} {g}_{1} u_{\emptyset,a} u_{c,b} u_{e,c} u_{a,d} u_{d,e}-{f}_{1} {g}_{1} u_{\emptyset,a} u_{e,b} u_{a,c} u_{c,d} u_{d,e}+{f}_{1} u_{\emptyset,a} u_{a,b} u_{g,c} u_{c,d} u_{d,e} u_{e,g}-{f}_{1} u_{\emptyset,a} u_{c,b} u_{g,c} u_{a,d} u_{d,e} u_{e,g}-{f}_{1} u_{\emptyset,a} u_{g,b} u_{a,c} u_{c,d} u_{d,e} u_{e,g}+{g}_{1} u_{\emptyset,a} u_{a,b} u_{f,c} u_{c,d} u_{d,e} u_{e,f}-{g}_{1} u_{\emptyset,a} u_{c,b} u_{f,c} u_{a,d} u_{d,e} u_{e,f}-{g}_{1} u_{\emptyset,a} u_{f,b} u_{a,c} u_{c,d} u_{d,e} u_{e,f}-u_{\emptyset,a} u_{a,b} u_{e,c} u_{c,d} u_{d,e} u_{g,f} u_{f,g}+u_{\emptyset,a} u_{a,b} u_{f,c} u_{c,d} u_{d,e} u_{g,f} u_{e,g}+u_{\emptyset,a} u_{a,b} u_{g,c} u_{c,d} u_{d,e} u_{e,f} u_{f,g}+u_{\emptyset,a} u_{c,b} u_{e,c} u_{a,d} u_{d,e} u_{g,f} u_{f,g}-u_{\emptyset,a} u_{c,b} u_{f,c} u_{a,d} u_{d,e} u_{g,f} u_{e,g}-u_{\emptyset,a} u_{c,b} u_{g,c} u_{a,d} u_{d,e} u_{e,f} u_{f,g}+u_{\emptyset,a} u_{e,b} u_{a,c} u_{c,d} u_{d,e} u_{g,f} u_{f,g}-u_{\emptyset,a} u_{f,b} u_{a,c} u_{c,d} u_{d,e} u_{g,f} u_{e,g}-u_{\emptyset,a} u_{g,b} u_{a,c} u_{c,d} u_{d,e} u_{e,f} u_{f,g}$,
\item ${b}_{1} {b}_{2}-{a}_{2} u_{a,b}-{c}_{2} u_{c,b}-{d}_{2} u_{d,b}-{e}_{2} u_{e,b}-{f}_{2} u_{f,b}-{g}_{2} u_{g,b}$,
\item ${a}_{1} {c}_{2}-{b}_{1} {d}_{1} u_{\emptyset,a} u_{a,c}-{b}_{1} {e}_{1} u_{\emptyset,a} u_{d,c} u_{a,d}-{b}_{1} {f}_{1} {g}_{1} u_{\emptyset,a} u_{e,c} u_{a,d} u_{d,e}-{b}_{1} {f}_{1} u_{\emptyset,a} u_{g,c} u_{a,d} u_{d,e} u_{e,g}-{b}_{1} {g}_{1} u_{\emptyset,a} u_{f,c} u_{a,d} u_{d,e} u_{e,f}+{b}_{1} u_{\emptyset,a} u_{e,c} u_{a,d} u_{d,e} u_{g,f} u_{f,g}-{b}_{1} u_{\emptyset,a} u_{f,c} u_{a,d} u_{d,e} u_{g,f} u_{e,g}-{b}_{1} u_{\emptyset,a} u_{g,c} u_{a,d} u_{d,e} u_{e,f} u_{f,g}-{d}_{1} u_{\emptyset,a} u_{a,b} u_{b,c}-{e}_{1} u_{\emptyset,a} u_{a,b} u_{d,c} u_{b,d}+{e}_{1} u_{\emptyset,a} u_{d,b} u_{a,c} u_{b,d}-{e}_{1} u_{\emptyset,a} u_{d,b} u_{b,c} u_{a,d}-{f}_{1} {g}_{1} u_{\emptyset,a} u_{a,b} u_{e,c} u_{b,d} u_{d,e}+{f}_{1} {g}_{1} u_{\emptyset,a} u_{e,b} u_{a,c} u_{b,d} u_{d,e}-{f}_{1} {g}_{1} u_{\emptyset,a} u_{e,b} u_{b,c} u_{a,d} u_{d,e}-{f}_{1} u_{\emptyset,a} u_{a,b} u_{g,c} u_{b,d} u_{d,e} u_{e,g}+{f}_{1} u_{\emptyset,a} u_{g,b} u_{a,c} u_{b,d} u_{d,e} u_{e,g}-{f}_{1} u_{\emptyset,a} u_{g,b} u_{b,c} u_{a,d} u_{d,e} u_{e,g}-{g}_{1} u_{\emptyset,a} u_{a,b} u_{f,c} u_{b,d} u_{d,e} u_{e,f}+{g}_{1} u_{\emptyset,a} u_{f,b} u_{a,c} u_{b,d} u_{d,e} u_{e,f}-{g}_{1} u_{\emptyset,a} u_{f,b} u_{b,c} u_{a,d} u_{d,e} u_{e,f}+u_{\emptyset,a} u_{a,b} u_{e,c} u_{b,d} u_{d,e} u_{g,f} u_{f,g}-u_{\emptyset,a} u_{a,b} u_{f,c} u_{b,d} u_{d,e} u_{g,f} u_{e,g}-u_{\emptyset,a} u_{a,b} u_{g,c} u_{b,d} u_{d,e} u_{e,f} u_{f,g}-u_{\emptyset,a} u_{e,b} u_{a,c} u_{b,d} u_{d,e} u_{g,f} u_{f,g}+u_{\emptyset,a} u_{e,b} u_{b,c} u_{a,d} u_{d,e} u_{g,f} u_{f,g}+u_{\emptyset,a} u_{f,b} u_{a,c} u_{b,d} u_{d,e} u_{g,f} u_{e,g}-u_{\emptyset,a} u_{f,b} u_{b,c} u_{a,d} u_{d,e} u_{g,f} u_{e,g}+u_{\emptyset,a} u_{g,b} u_{a,c} u_{b,d} u_{d,e} u_{e,f} u_{f,g}-u_{\emptyset,a} u_{g,b} u_{b,c} u_{a,d} u_{d,e} u_{e,f} u_{f,g}$,
\item ${c}_{1} {c}_{2}-{a}_{2} u_{a,c}-{b}_{2} u_{b,c}-{d}_{2} u_{d,c}-{e}_{2} u_{e,c}-{f}_{2} u_{f,c}-{g}_{2} u_{g,c}$,
\item ${a}_{1} {d}_{2}-{b}_{1} {c}_{1} {e}_{1} u_{\emptyset,a} u_{a,d}-{b}_{1} {e}_{1} u_{\emptyset,a} u_{a,c} u_{c,d}-{c}_{1} {e}_{1} u_{\emptyset,a} u_{a,b} u_{b,d}-{e}_{1} u_{\emptyset,a} u_{a,b} u_{b,c} u_{c,d}-{e}_{1} u_{\emptyset,a} u_{c,b} u_{a,c} u_{b,d}+{e}_{1} u_{\emptyset,a} u_{c,b} u_{b,c} u_{a,d}$,
\item ${d}_{1} {d}_{2}-{e}_{1} {a}_{2} u_{a,d}-{e}_{1} {b}_{2} u_{b,d}-{e}_{1} {c}_{2} u_{c,d}$,
\item ${a}_{1} {e}_{2}-{b}_{1} {c}_{1} {f}_{1} {g}_{1} u_{\emptyset,a} u_{a,d} u_{d,e}+{b}_{1} {c}_{1} u_{\emptyset,a} u_{a,d} u_{d,e} u_{g,f} u_{f,g}-{b}_{1} {f}_{1} {g}_{1} u_{\emptyset,a} u_{a,c} u_{c,d} u_{d,e}+{b}_{1} u_{\emptyset,a} u_{a,c} u_{c,d} u_{d,e} u_{g,f} u_{f,g}-{c}_{1} {f}_{1} {g}_{1} u_{\emptyset,a} u_{a,b} u_{b,d} u_{d,e}+{c}_{1} u_{\emptyset,a} u_{a,b} u_{b,d} u_{d,e} u_{g,f} u_{f,g}-{f}_{1} {g}_{1} u_{\emptyset,a} u_{a,b} u_{b,c} u_{c,d} u_{d,e}-{f}_{1} {g}_{1} u_{\emptyset,a} u_{c,b} u_{a,c} u_{b,d} u_{d,e}+{f}_{1} {g}_{1} u_{\emptyset,a} u_{c,b} u_{b,c} u_{a,d} u_{d,e}+u_{\emptyset,a} u_{a,b} u_{b,c} u_{c,d} u_{d,e} u_{g,f} u_{f,g}+u_{\emptyset,a} u_{c,b} u_{a,c} u_{b,d} u_{d,e} u_{g,f} u_{f,g}-u_{\emptyset,a} u_{c,b} u_{b,c} u_{a,d} u_{d,e} u_{g,f} u_{f,g}$,
\item ${d}_{1} {e}_{2}-{f}_{1} {g}_{1} {a}_{2} u_{a,d} u_{d,e}-{f}_{1} {g}_{1} {b}_{2} u_{b,d} u_{d,e}-{f}_{1} {g}_{1} {c}_{2} u_{c,d} u_{d,e}+{a}_{2} u_{a,d} u_{d,e} u_{g,f} u_{f,g}+{b}_{2} u_{b,d} u_{d,e} u_{g,f} u_{f,g}+{c}_{2} u_{c,d} u_{d,e} u_{g,f} u_{f,g}$,
\item ${e}_{1} {e}_{2}-{f}_{1} {g}_{1} {d}_{2} u_{d,e}+{d}_{2} u_{d,e} u_{g,f} u_{f,g}$,
\item ${a}_{1} {f}_{2}-{b}_{1} {c}_{1} {g}_{1} u_{\emptyset,a} u_{a,d} u_{d,e} u_{e,f}-{b}_{1} {c}_{1} u_{\emptyset,a} u_{a,d} u_{d,e} u_{g,f} u_{e,g}-{b}_{1} {g}_{1} u_{\emptyset,a} u_{a,c} u_{c,d} u_{d,e} u_{e,f}-{b}_{1} u_{\emptyset,a} u_{a,c} u_{c,d} u_{d,e} u_{g,f} u_{e,g}-{c}_{1} {g}_{1} u_{\emptyset,a} u_{a,b} u_{b,d} u_{d,e} u_{e,f}-{c}_{1} u_{\emptyset,a} u_{a,b} u_{b,d} u_{d,e} u_{g,f} u_{e,g}-{g}_{1} u_{\emptyset,a} u_{a,b} u_{b,c} u_{c,d} u_{d,e} u_{e,f}-{g}_{1} u_{\emptyset,a} u_{c,b} u_{a,c} u_{b,d} u_{d,e} u_{e,f}+{g}_{1} u_{\emptyset,a} u_{c,b} u_{b,c} u_{a,d} u_{d,e} u_{e,f}-u_{\emptyset,a} u_{a,b} u_{b,c} u_{c,d} u_{d,e} u_{g,f} u_{e,g}-u_{\emptyset,a} u_{c,b} u_{a,c} u_{b,d} u_{d,e} u_{g,f} u_{e,g}+u_{\emptyset,a} u_{c,b} u_{b,c} u_{a,d} u_{d,e} u_{g,f} u_{e,g}$,
\item ${d}_{1} {f}_{2}-{g}_{1} {a}_{2} u_{a,d} u_{d,e} u_{e,f}-{g}_{1} {b}_{2} u_{b,d} u_{d,e} u_{e,f}-{g}_{1} {c}_{2} u_{c,d} u_{d,e} u_{e,f}-{a}_{2} u_{a,d} u_{d,e} u_{g,f} u_{e,g}-{b}_{2} u_{b,d} u_{d,e} u_{g,f} u_{e,g}-{c}_{2} u_{c,d} u_{d,e} u_{g,f} u_{e,g}$,
\item ${e}_{1} {f}_{2}-{g}_{1} {d}_{2} u_{d,e} u_{e,f}-{d}_{2} u_{d,e} u_{g,f} u_{e,g}$,
\item ${f}_{1} {f}_{2}-{e}_{2} u_{e,f}-{g}_{2} u_{g,f}$,
\item ${a}_{1} {g}_{2}-{b}_{1} {c}_{1} {f}_{1} u_{\emptyset,a} u_{a,d} u_{d,e} u_{e,g}-{b}_{1} {c}_{1} u_{\emptyset,a} u_{a,d} u_{d,e} u_{e,f} u_{f,g}-{b}_{1} {f}_{1} u_{\emptyset,a} u_{a,c} u_{c,d} u_{d,e} u_{e,g}-{b}_{1} u_{\emptyset,a} u_{a,c} u_{c,d} u_{d,e} u_{e,f} u_{f,g}-{c}_{1} {f}_{1} u_{\emptyset,a} u_{a,b} u_{b,d} u_{d,e} u_{e,g}-{c}_{1} u_{\emptyset,a} u_{a,b} u_{b,d} u_{d,e} u_{e,f} u_{f,g}-{f}_{1} u_{\emptyset,a} u_{a,b} u_{b,c} u_{c,d} u_{d,e} u_{e,g}-{f}_{1} u_{\emptyset,a} u_{c,b} u_{a,c} u_{b,d} u_{d,e} u_{e,g}+{f}_{1} u_{\emptyset,a} u_{c,b} u_{b,c} u_{a,d} u_{d,e} u_{e,g}-u_{\emptyset,a} u_{a,b} u_{b,c} u_{c,d} u_{d,e} u_{e,f} u_{f,g}-u_{\emptyset,a} u_{c,b} u_{a,c} u_{b,d} u_{d,e} u_{e,f} u_{f,g}+u_{\emptyset,a} u_{c,b} u_{b,c} u_{a,d} u_{d,e} u_{e,f} u_{f,g}$,
\item ${d}_{1} {g}_{2}-{f}_{1} {a}_{2} u_{a,d} u_{d,e} u_{e,g}-{f}_{1} {b}_{2} u_{b,d} u_{d,e} u_{e,g}-{f}_{1} {c}_{2} u_{c,d} u_{d,e} u_{e,g}-{a}_{2} u_{a,d} u_{d,e} u_{e,f} u_{f,g}-{b}_{2} u_{b,d} u_{d,e} u_{e,f} u_{f,g}-{c}_{2} u_{c,d} u_{d,e} u_{e,f} u_{f,g}$,
\item ${e}_{1} {g}_{2}-{f}_{1} {d}_{2} u_{d,e} u_{e,g}-{d}_{2} u_{d,e} u_{e,f} u_{f,g}$,
\item ${g}_{1} {g}_{2}-{e}_{2} u_{e,g}-{f}_{2} u_{f,g}$,
\end{enumerate}
\end{footnotesize}

\bibliographystyle{amsplain}
\bibliography{BibliographyA}

\end{document}